\documentclass[12pt]{article}
\usepackage[a4paper, left=1.5cm, right=1.5cm]{geometry}
\usepackage{hyperref}
\usepackage{graphicx}
\usepackage{multirow}
\usepackage{amsmath,amssymb,amsfonts}
\usepackage{amsthm}
\usepackage{mathrsfs}
\usepackage{xcolor}
\usepackage{enumitem}
\usepackage{array}
\usepackage{mathdots}
\usepackage{mathtools}
\usepackage{tikz}
\usepackage{anyfontsize}

\newcommand{\FF}{{\mathbb{F}}}

\newcommand{\ZZ}{\mathbb{Z}}

\newcommand{\supp}{\mathrm{supp}}
\newcommand{\wt}{\mathrm{wt}}
\newcommand{\swe}{\mathrm{swe}}
\newcommand{\sse}{\mathrm{sse}}

\newcommand{\cv}{\mathbf{c}}

\newcommand{\gv}{\mathbf{g}}

\newcommand{\uv}{\mathbf{u}}
\newcommand{\vv}{\mathbf{v}}
\newcommand{\wv}{\mathbf{w}}
\newcommand{\xv}{\mathbf{x}}
\newcommand{\yv}{\mathbf{y}}
\newcommand{\zv}{\mathbf{z}}

\newtheorem{theorem}{Theorem}
\newtheorem{prop}[theorem]{Proposition}
\newtheorem{lem}[theorem]{Lemma}
\newtheorem{cor}[theorem]{Corollary}

\theoremstyle{definition}%
\newtheorem{ex}{Example}%
\newtheorem{rem}{Remark}%

\title{A multiset approach to MacWilliams identities}
\author{Hopein Christofen Tang\thanks{School of Mathematics and Statistics, UNSW Sydney, NSW 2052, Australia, hopein.tang@unsw.edu.au}}
\date{}

\begin{document}

\maketitle
\abstract{We interpret the symmetrized weight enumerator of linear codes over finite commutative Frobenius rings as a summation over multisets and thereby provide a new proof of the MacWilliams identity for the symmetrized weight enumerator. The proof and the identity are expressed in combinatorial terms that do not require generating characters. We also generalize the symmetrized weight enumerator with respect to supports and codeword tuples, and our multiset approach enables us to derive new and general MacWilliams identities expressed in combinatorial terms.}

\maketitle

\section{Introduction}\label{sec1}

\noindent One of the most well-known and important results in coding theory is the classic MacWilliams identity \cite{macwilliams63}, which relates the Hamming weight enumerator of a code over a finite field to that of its dual code. The MacWilliams identity has been generalized with respect to many other enumerators over various finite rings; see \cite{DoSk2002,GaYa08,HoLa01,Klove,Wood1999} for some of the well-known generalizations. In this paper, we specifically consider the symmetrized weight enumerator of linear codes over finite Frobenius rings, which is one of the generalizations of the Hamming weight enumerator considered in \cite{Wood1999}. The MacWilliams identity with respect to this enumerator was proven by Wood \cite{Wood1999} via a character-theoretical approach and expressed in terms of sums involving the generating character of the ring. For some relatively recent applications of Wood's MacWilliams identity for the symmetrized weight enumerator with respect to specific finite Frobenius rings, see \cite{CaAs20, CeDeDo14, DoSaSz19, IrDj18, IrDj23}.

The main purpose of this paper is to offer a new method for obtaining MacWilliams identities for the symmetrized weight enumerator of linear codes over finite commutative Frobenius rings. Our approach is based on the observation that the symmetrized weight enumerator can be expressed as a summation over multisets. This approach enables us to obtain a new proof of the MacWilliams identity for the symmetrized weight enumerator without involving generating characters; see Theorem~\ref{thm:main}. Our main result in Theorem~\ref{thm:main} reveals a previously unknown relationship between the MacWilliams identity and the structural aspects of the poset of principal ideals of the ring, such as its adjacency matrix and the cardinalities of the principal ideals. In particular, we show that it is possible to obtain the MacWilliams identity for the symmetrized weight enumerator solely from three matrices related to the poset of principal ideals of the ring and their dual; the generating character of the ring need not be known.  

The paper is organized as follows. We provide basic definitions and notations in Section~\ref{sec:Preliminaries}. In Section~\ref{sec:main}, we introduce our multiset approach and prove the main result of the paper (Theorem~\ref{thm:main}).  
In Section~\ref{sec:applications}, we apply our MacWilliams identity in Theorem~\ref{thm:main} to finite commutative chain rings and principal ideal rings, and obtain more specific results; see Corollaries~\ref{cor:FCR} and~\ref{cor:FPIR}. We also present examples to show that our MacWilliams identity in Theorem~\ref{thm:main} yields the same specific identities as the ones obtained from Wood's result~\cite{Wood1999}; see Examples~\ref{ex:Z4/<x^3-2,2x>}, \ref{ex:Z12} and \ref{ex:F2[u,v]}. Furthermore, we 
derive some of the known MacWilliams identities for other weight enumerators. In particular, we consider the MacWilliams identity for the (symmetrized) Lee weight enumerator discussed in \cite{LiuLiu15} and give a new proof of the MacWilliams identity for the Hamming weight enumerator; 
see Examples~\ref{ex:Lee} and \ref{ex:Hamming}. 

In Section~\ref{sec:generalizations}, we consider several generalizations of the symmetrized weight enumerator with respect to symmetrized supports and codeword tuples; see \cite{britz02,BrShWe15,shiromoto96} for similar generalizations of the Hamming weight enumerator and \cite{Kaplan2014, Siap1999, SiRC2000} for further known generalizations. By applying our multiset approach to these general enumerators, we derive general MacWilliams identities that do not require generating characters. These identities generalize Wood's identity~\cite{Wood1999} and Theorem~\ref{thm:main}; see Theorems~\ref{thm:main-support} and \ref{thm:main-tuple} and Corollary~\ref{cor:SWE-tuple}. To the best of our knowledge, the MacWilliams identities in Theorems~\ref{thm:main-support} and \ref{thm:main-tuple} are new. 

\section{Preliminaries}
\label{sec:Preliminaries}
\subsection{Linear codes over finite Frobenius rings}
Throughout this paper, let $R$ denote a finite ring. For $a\in R$, define $aR:=\{ab : b\in R\}$ and $aR^\times:=\{ab : b\in R^\times\}$, where $R^\times$ denotes the set of all units of $R$. 
A~\emph{code} $C$ of length $n$ over $R$ is a non-empty subset of~$R^n$. The~code $C$ is \emph{linear} if it is a submodule of~$R^n$. A~matrix $\mathbf{G}$ over $R$ is a \emph{generator matrix} of a linear code~$C$ over $R$ if its rows span~$C$. The \emph{inner product} of two codewords $\uv=(u_1,\ldots,u_n)$ and $\vv=(v_1,\ldots,v_n)$ in $R^n$
is 
$\uv\cdot\vv:=\sum_{\ell=1}^n u_\ell v_\ell\in R$ and the \emph{dual} $C^\perp$ of a code $C$ over $R$ is the linear code 
\[
  C^\perp:=\{\vv \in R^n:~\uv \cdot \vv=0 \text{ for all }\uv \in C\}\,.
\]
A finite ring $R$ is \emph{Frobenius} if $R/\mathrm{Rad}(R)\cong \mathrm{Soc}(_R R)$ as left $R$-modules and $R/\mathrm{Rad}(R)\cong \mathrm{Soc}(R_R)$ as right $R$-modules, where $\mathrm{Rad}(R)$ is the Jacobson radical of $R$ and $\mathrm{Soc}(_R M)$ (resp. $\mathrm{Soc}(M_R)$)  denotes the socle of a left (resp.\ right) $R$-module $M$; see \cite{Lam1999}. There are several known characterizations of finite Frobenius rings. In this paper, we consider the following well-known characterization; see \cite{BrShWe15,Honold2001,HoLa01,Wood1999,Wood2009}.
\begin{lem}\label{lem:Wood}
    A finite ring $R$ is Frobenius if and only if $|C|\cdot|C^\perp|=R^n$ for any linear code~$C$ of length $n$ over $R$.
\end{lem}
\noindent This property gives us the following useful corollary.
\begin{cor}\label{cor:span}
    Let $C$ be a linear code of length $n$ over a finite commutative Frobenius ring~$R$, and let $\mathbf{G}=(\gv_1 \cdots \gv_n)\in R^{m\times n}$ be a generator matrix of~$C$. Then 
    \[|\mathrm{span}\{\gv_1,\dots,\gv_n\}|=|C|\,.\]
\end{cor}
\begin{proof}
    Consider a module homomorphism 
    \begin{eqnarray*}
  \pi: & R^n &\rightarrow \quad \;\;R^m\\
          & (c_1,\dots,c_n)                                   &\mapsto\;\;\sum_{\ell=1}^n c_\ell \gv_\ell\,,
\end{eqnarray*}
and observe that $\mathrm{im}(\pi)=\mathrm{span}\{\gv_1,\dots,\gv_n\}$. It is not hard to show that $\ker(\pi)=C^\perp$. The result follows from the Module Isomorphism Theorem and Lemma~\ref{lem:Wood}.
\end{proof}
\noindent It also follows immediately from Lemma~\ref{lem:Wood} that the dual of the dual of a linear code~$C$ over a finite commutative Frobenius ring $R$ is $C$ since $C\subseteq (C^\perp)^\perp$ by definition.
\begin{lem}\label{lem:AA}
    Let $C$ be a linear code over a finite commutative Frobenius ring~$R$. Then $(C^\perp)^\perp=C$.
\end{lem}
\noindent The next corollary follows from the definition of the dual and the lemma above.
\begin{cor}\label{cor:dual}
    Let $R$ be a finite commutative Frobenius ring. Then for any $c\in R$,
    \[(cR)^\perp=\{r\in R\;\colon\;cr=0\}\,.\]
    Moreover, for any $a,b\in R$, $aR\subseteq bR$ if and only if $(bR)^\perp\subseteq (aR)^\perp$.
\end{cor}

\subsection{Hamming support and Hamming weight enumerator}
The \emph{Hamming support} of $\cv=(c_1,\ldots,c_n)\in~R^n$ is
\[
  \supp(\cv):=\{\ell\in [n] : c_\ell\neq 0\}\,,
\]
where $[n]:=\{1,\dots,n\}$. The \emph{Hamming weight} of $\cv\!\in \!R^n$ is $\wt(\cv):=|\supp(\cv)|$ and the \emph{Hamming weight enumerator} of a linear code $C$ of length $n$ over $R$ is
\begin{equation*}
    W_C(x,y):=\sum_{\cv\in C}x^{n-\wt(\cv)}y^{\wt(\cv)}.
\end{equation*}
It is well-known that the Hamming weight enumerator of a linear code $C$ over a finite Frobenius ring $R$ is related to the Hamming weight enumerator of $C^\perp$ via the MacWilliams identity as follows.
\begin{theorem}\label{thm:MW-HWE-FFR} {\rm \cite{ BrShWe15, HoLa01, Wood1999}}
    Let $C$ be a linear code over a finite Frobenius ring $R$. Then
    \begin{equation}\label{eq:MW-HWE}
        W_{C^\perp}(x,y)=\frac{1}{|C|}W_C(x+(|R|-1)y,x-y)\,.
    \end{equation}
\end{theorem}

\subsection{Symmetrized support and symmetrized weight enumerator}\label{subsection:swe}
We define the symmetrized weight enumerator following the definition in \cite{Wood1999}. Define an equivalence relation $\approx$ on a finite commutative Frobenius ring $R$ by $r\approx s$ if $r=us$ for some unit $u\in R$. Note that $r\approx s$ if and only if $rR=sR$; see \cite[Proposition 5.1]{Wood1999}. Throughout the paper, let $t$ denote the number of nonzero principal ideals of $R$. Since we know that $R$ has exactly $t+1$ equivalence classes with respect to the relation $\approx$, we can find a set of representatives $\{a_0,a_1,\dots,a_t\}\subseteq R$ such that $a_i R\neq a_j R$ for any distinct $i,j\in\{0,1,\dots,t\}$. Throughout this paper, $a_0, a_1,\dots,a_t$ will always denote these representatives.

Now for each $i\in\{0,1,\dots,t\}$ and $\cv=(c_1,\ldots,c_n)\in~R^n$, define
\begin{equation}\label{eq:def-Si}
  S_i(\cv):=\{\ell\in [n] : c_\ell\approx a_i\}=\{\ell\in [n] : c_\ell R= a_i R\}\,.  
\end{equation}
We will refer to $S_0(\cv),S_1(\cv),\dots,S_t(\cv)$ as the \emph{symmetrized supports} of $\cv$. Additionally, for the rest of this paper, let
\[\wv:=\begin{pmatrix}
    w_0\\w_1\\\vdots\\ w_t
\end{pmatrix},\quad \xv:=\begin{pmatrix}
    x_0\\x_1\\\vdots\\ x_t
\end{pmatrix},\quad\yv:=\begin{pmatrix}
    y_0\\y_1\\\vdots\\ y_t
\end{pmatrix},\quad\zv:=\begin{pmatrix}
    z_0\\z_1\\\vdots\\ z_t
\end{pmatrix}\quad\]
be vectors of indeterminates indexed by $\{0,1,\dots,t\}$. The \emph{symmetrized weight enumerator} of a linear code $C$ of length $n$ over a finite commutative Frobenius ring $R$ is defined as
\begin{equation}\label{eq:SWE}
    \swe_C(\xv)=\swe_C(x_0,x_1,\dots,x_t)
    :=\sum_{\cv\in C} x_0^{\mathrm{swc}_0(\cv)}x_1^{\mathrm{swc}_1(\cv)}\cdots x_t^{\mathrm{swc}_t(\cv)},
\end{equation}
where for each $i\in\{0,1,\dots,t\}$, \[\mathrm{swc}_i(\cv):=|S_i(\cv)|\]
is the \emph{symmetrized weight composition} of $\cv$. Note that if $a_0=0$, then
\begin{equation}\label{eq:HWE-SWE}
    W_C(x,y)=\swe_C(x,y,\dots,y).
\end{equation}

It is known from \cite{Wood1999} that the following MacWilliams identity holds for the symmetrized weight enumerator of linear codes over finite commutative Frobenius rings.
\begin{theorem}\label{thm:MW-Wood} {\rm\cite[Theorem 8.4]{Wood1999}}
    Let $C$ be a linear code of length $n$ over a finite commutative Frobenius ring~$R$, and let $\chi:R\rightarrow\mathbb{C}$ be a generating character of $R$. Moreover, let $\mathbf{S}=(s_{ij})$ be a $(t+1)\times (t+1)$ matrix whose rows and columns are indexed by $\{0,1,\dots,t\}$, where $s_{ij}:=\sum_{r\,\approx\,a_j}\chi(ra_i).$
    Then
    \[
    \swe_{C^\perp}(\xv)=\frac{1}{|C|}\swe_C(\mathbf{S}\xv)\,.
    \]
\end{theorem}
\noindent For the exact definition of a generating character, see \cite{Wood1999}. For a constructive method for finding the generating character of any finite commutative Frobenius ring, see \cite{DoSaSz19}. 

\section{MacWilliams identity for the symmetrized weight enumerator}\label{sec:main}
In this section, we will prove our main result, namely a MacWilliams identity for the symmetrized weight enumerator that does not involve generating characters. Our main approach is to look at the symmetrized weight enumerator as a sum over submultisets. For~$t$ as defined in Subsection~\ref{subsection:swe}, let $t\cdot[n]$ denote the multiset with underlying set~$[n]$ in which each of $1,\dots, n$ appears exactly $t$ times. For each multiset $X\subseteq t\cdot[n]$, let $m_X(\ell)$ denote the \emph{multiplicity} of~$\ell\in [n]$ in $X$, namely the number of occurrences of $\ell$ in the multiset~$X$. For each integer $i\in\{0,1,\dots,t\}$,~define the set
\begin{align*}
    M_i(X)&:=\{\ell\in [n]\;\colon\;m_X(\ell)=i\}\,.
\end{align*}
Now let $C$ be a linear code of length $n$ over a finite commutative Frobenius ring~$R$. Since $S_0(\cv),S_1(\cv),\dots,S_t(\cv)$  are pairwise disjoint sets whose union is $[n]$ by definition~(\ref{eq:def-Si}), we are able to express $\swe_C$ in~(\ref{eq:SWE}) as a sum over submultisets of $t\cdot[n]$. In particular,
\begin{equation}\label{eq:SWE-multiset}
    \swe_C(\xv)=\sum_{X\subseteq t\cdot[n]} A_C(X) x_0^{|M_0(X)|}x_1^{|M_1(X)|}\cdots x_t^{|M_t(X)|}\,,
\end{equation}
where for any multiset $X\subseteq t\cdot[n]$,
\begin{equation}\label{eq:AC}
    A_C(X):=|\{\cv\in C\;\colon\; S_i(\cv)=M_i(X) \text{ for all }i=0,1,\dots,t\}|\,.
\end{equation}
It is worth noting that the condition $S_i(\cv)=M_i(X)$ for all $i=0,1,\dots,t$ means that the multiset $X$ contains the element $\ell\in[n]$ exactly $i$ times when the $\ell$-th coordinate of the codeword $\cv=(c_1,\dots,c_n)$ generates $a_i R$, i.e., $c_\ell R=a_i R$.  

\begin{rem}
    While our main approach is to look at the symmetrized weight enumerator as a sum over multisets, similar approaches can be applied to other combinatorial objects that are equivalent to multisets, for example chains of subsets.
\end{rem}
For each multiset $X\subseteq t\cdot[n]$ and sets $F_0,F_1,\dots,F_t\subseteq\{0,1,\dots,t\}$, let $\mathcal{F}:=(F_0,F_1,\dots,F_t)$ and define
\begin{equation}\label{eq:BC}
    B^{\mathcal{F}}_C(X):=\Big|\Big\{\cv\in C\;\colon\; S_i(\cv)\subseteq \bigcup_{j\in F_i}M_j(X) \text{ for all }i=0,1,\dots,t\Big\}\Big|\,.
\end{equation}

\noindent We can replace $A_C$ in (\ref{eq:SWE-multiset}) by $B^{\mathcal{F}}_{C}$ and obtain a new expression in terms of $\swe_C$.
\begin{theorem}\label{thm:id-SWE}
    Let $C$ be a linear code of length $n$ over a finite commutative Frobenius ring~$R$, and let $\mathcal{F}:=(F_0,F_1,\dots,F_t)$, where $F_0,F_1,\dots,F_t$ are non-empty subsets of $\{0,1,\dots,t\}$. Then
    \begin{equation}\label{eq:id-SWE}
        \swe_C\big(\mathbf{P}^\mathcal{F}\yv\big)=\sum_{Y\subseteq t\cdot[n]} B^{\mathcal{F}}_{C}(Y) y_0^{|M_0(Y)|}y_1^{|M_1(Y)|}\cdots y_t^{|M_t(Y)|}\,,
    \end{equation}
    where $\mathbf{P}^\mathcal{F}=(p_{ij}^\mathcal{F})$ is the $(t+1)\times (t+1)$ matrix whose rows and columns are indexed by $\{0,1,\dots,t\}$ with 
    \[p^\mathcal{F}_{ij}:=\begin{cases}
    1,&\text{ if }j\in F_i\,;\\
    0,&\text{ otherwise.}
    \end{cases}\] 
\end{theorem}
\begin{proof}
Throughout the proof, define for any set $T\subseteq [n]$ and each positive integer $m$,
\[\mathcal{P}_m(T):=\Big\{(T_1,\dots,T_m)\;\colon\;\bigcup_{i=1}^m T_i = T\text{ and }T_i\cap T_j=\emptyset\text{ for all }i\neq j\Big\}\,.\]
Then the following identity holds:
\begin{equation}\label{eq:id-part}
    \sum_{(T_1,\dots,T_m)\in\mathcal{P}_m(T)} u_1^{|T_1|}\cdots u_m^{|T_m|}=(u_1+\cdots+u_m)^{|T|}\,.
\end{equation}
For any non-empty sets of integers $V$ and $T\subseteq [n]$, let $(T_j\;\colon\;j\in V)$ denote a $|V|$-tuple of the sets $T_j\subseteq T$ indexed by the elements of $V$ with respect to the usual ordering of integers. Then~(\ref{eq:id-part}) can be generalized to
\begin{equation}\label{eq:id-part-set}
    \sum_{(T_j\,\colon\,j\,\in\, V)\in\mathcal{P}_{|V|}(T)}\Big(\prod_{j\in V} u_j^{|T_j|}\Big)=\Big(\sum_{j\in V} u_j\Big)^{|T|}.
\end{equation}
Now consider a binary relation $\leq$ over $t\cdot[n]$, where for any $X, Y\subseteq t\cdot[n]$,
\[X\leq Y\;\text{ if }\; M_i(X)\subseteq \bigcup_{j\in F_i}M_j(Y)\;\text{ for all }i=0,1,\dots,t.\]
Note that $\leq$ is not necessarily reflexive, symmetric, anti-symmetric or transitive. However,
\begin{equation*}\label{eq:proof-id}
  B^{\mathcal{F}}_{C}(Y)=\sum_{X\leq Y}A_C(X)\,.  
\end{equation*}
This fact allows us to simplify the expression in the right-hand side of (\ref{eq:id-SWE}) in terms of $A_C$. Observe that
\begin{align}\label{eq:proof1}
    \sum_{Y\subseteq t\cdot[n]} B^{\mathcal{F}}_{C}(Y)\prod_{j=0}^t y_j^{|M_j(Y)|}&=\hspace*{-.5mm}\sum_{Y\subseteq t\cdot[n]} \sum_{X\leq Y}A_C(X)\prod_{j=0}^t y_j^{|M_j(Y)|}\notag\\&=\hspace*{-.5mm}\sum_{X\subseteq t\cdot[n]}\hspace*{-1mm}A_C(X)\hspace*{-1mm}\sum_{\substack{Y\subseteq t\cdot[n]\colon\\X\leq Y}}\prod_{j=0}^t y_j^{|M_j(Y)|}.
\end{align}
To simplify (\ref{eq:proof1}) further, define for any $i,j\in\{0,1,\dots,t\}$,
\[W_{ij}:=M_i(X)\cap M_j(Y)\,.\]
Since $X\leq Y$, $W_{ij}=\emptyset$ for all $j\not\in F_i$ and thus $(W_{ij}\,\colon\,j\in F_i)\in \mathcal{P}_{|F_i|}(M_i(X))$. Moreover,
\begin{equation}\label{eq:proof2}
    |M_j(Y)|=\sum_{\substack{i\in\{0,1,\dots,t\}\colon\\j\in F_i}}\hspace*{-2mm}|W_{ij}|\,.
\end{equation}
Therefore, from (\ref{eq:proof2}),
\begingroup
	\allowdisplaybreaks
\begin{align*}
    &\sum_{\substack{Y\subseteq t\cdot[n]\colon\\X\leq Y}}\prod_{j=0}^t y_j^{|M_j(Y)|}\\
    &=\sum_{\substack{(\hspace*{.3mm}W_{\hspace*{.25mm}0\hspace*{.15mm}j\hspace*{.25mm}}\,\colon\,j\,\in\, F_0\hspace*{.25mm})\\\in \hspace*{.2mm}\mathcal{P}_{|\hspace*{-.2mm}F_0\hspace*{-.2mm}|}\hspace*{-.25mm}(\hspace*{-.1mm}M_0(X)\hspace*{-.1mm})}}
    \;\sum_{\substack{(\hspace*{.3mm}W_{\hspace*{.25mm}1\hspace*{.1mm}j\hspace*{.25mm}}\,\colon\,j\,\in\, F_1\hspace*{.25mm})\\\in \hspace*{.2mm}\mathcal{P}_{|\hspace*{-.2mm}F_1\hspace*{-.2mm}|}\hspace*{-.25mm}(\hspace*{-.1mm}M_1(X)\hspace*{-.1mm})}}
    \cdots
    \;\sum_{\substack{(\hspace*{.3mm}W_{\hspace*{.25mm}t\hspace*{.1mm}j\hspace*{.25mm}}\,\colon\,j\,\in\, F_t\hspace*{.25mm})\\\in \hspace*{.2mm}\mathcal{P}_{|\hspace*{-.2mm}F_t\hspace*{-.2mm}|}\hspace*{-.25mm}(\hspace*{-.1mm}M_t(X)\hspace*{-.1mm})}}\Bigg(\prod_{i=0}^t\prod_{j\in F_i} y_j^{|W_{ij}|}\Bigg)\\
    &=\sum_{\substack{(\hspace*{.3mm}W_{\hspace*{.25mm}0\hspace*{.15mm}j\hspace*{.25mm}}\,\colon\,j\,\in\, F_0\hspace*{.25mm})\\\in \hspace*{.2mm}\mathcal{P}_{|\hspace*{-.2mm}F_0\hspace*{-.2mm}|}\hspace*{-.25mm}(\hspace*{-.1mm}M_0(X)\hspace*{-.1mm})}}\,\prod_{j\in F_0} y_j^{|W_{0j}|}\Bigg(
    \;\sum_{\substack{(\hspace*{.3mm}W_{\hspace*{.25mm}1\hspace*{.1mm}j\hspace*{.25mm}}\,\colon\,j\,\in\, F_1\hspace*{.25mm})\\\in \hspace*{.2mm}\mathcal{P}_{|\hspace*{-.2mm}F_1\hspace*{-.2mm}|}\hspace*{-.25mm}(\hspace*{-.1mm}M_1(X)\hspace*{-.1mm})}}\,\prod_{j\in F_1} y_j^{|W_{1j}|}\Bigg(
    \cdots\Bigg(
    \;\sum_{\substack{(\hspace*{.3mm}W_{\hspace*{.25mm}t\hspace*{.1mm}j\hspace*{.25mm}}\,\colon\,j\,\in\, F_t\hspace*{.25mm})\\\in \hspace*{.2mm}\mathcal{P}_{|\hspace*{-.2mm}F_t\hspace*{-.2mm}|}\hspace*{-.25mm}(\hspace*{-.1mm}M_t(X)\hspace*{-.1mm})}}\,\prod_{j\in F_t} y_j^{|W_{tj}|}\Bigg)\cdots\Bigg)\Bigg)\\
    &=\Bigg(\sum_{\substack{(\hspace*{.3mm}W_{\hspace*{.25mm}0\hspace*{.15mm}j\hspace*{.25mm}}\,\colon\,j\,\in\, F_0\hspace*{.25mm})\\\in \hspace*{.2mm}\mathcal{P}_{|\hspace*{-.2mm}F_0\hspace*{-.2mm}|}\hspace*{-.25mm}(\hspace*{-.1mm}M_0(X)\hspace*{-.1mm})}}\,\prod_{j\in F_0} y_j^{|W_{0j}|}\Bigg)
    \;\Bigg(\sum_{\substack{(\hspace*{.3mm}W_{\hspace*{.25mm}1\hspace*{.1mm}j\hspace*{.25mm}}\,\colon\,j\,\in\, F_1\hspace*{.25mm})\\\in \hspace*{.2mm}\mathcal{P}_{|\hspace*{-.2mm}F_1\hspace*{-.2mm}|}\hspace*{-.25mm}(\hspace*{-.1mm}M_1(X)\hspace*{-.1mm})}}\,\prod_{j\in F_1} y_j^{|W_{1j}|}\Bigg)
    \cdots\Bigg(
    \;\sum_{\substack{(\hspace*{.3mm}W_{\hspace*{.25mm}t\hspace*{.1mm}j\hspace*{.25mm}}\,\colon\,j\,\in\, F_t\hspace*{.25mm})\\\in \hspace*{.2mm}\mathcal{P}_{|\hspace*{-.2mm}F_t\hspace*{-.2mm}|}\hspace*{-.25mm}(\hspace*{-.1mm}M_t(X)\hspace*{-.1mm})}}\,\prod_{j\in F_t} y_j^{|W_{tj}|}\Bigg)\\
    &=\prod_{i=0}^t\Bigg(\sum_{\substack{(\hspace*{.3mm}W_{\hspace*{.25mm}i\hspace*{.15mm}j\hspace*{.25mm}}\,\colon\,j\,\in\, F_i\hspace*{.25mm})\\\in \hspace*{.2mm}\mathcal{P}_{|\hspace*{-.2mm}F_i\hspace*{-.2mm}|}\hspace*{-.25mm}(\hspace*{-.1mm}M_i(X)\hspace*{-.1mm})}}\,\prod_{j\in F_i} y_i^{|W_{ij}|}\Bigg).
\end{align*}
\endgroup
We can simpify the last expression above using (\ref{eq:id-part-set}) to obtain
\[
    \sum_{\substack{Y\subseteq t\cdot[n]\colon\\X\leq Y}}\prod_{j=0}^t y_j^{|M_j(Y)|}=\prod_{i=0}^t\Big(\sum_{j\in F_i}y_j\Big)^{|M_i(X)|}=\prod_{i=0}^t\Big(\sum_{j=0}^t p^\mathcal{F}_{ij} y_j\Big)^{|M_i(X)|}.\]
Substituting the equation above into (\ref{eq:proof1}) and comparing with the form of $\swe_C$ in (\ref{eq:SWE-multiset}) completes the proof.
\end{proof}
\noindent The identity in (\ref{eq:id-SWE}) shows that it is possible to relate $\swe_{C^\perp}$ and $\swe_C$ by finding suitable families $\mathcal{I}:=(I_0,I_1,\dots,I_t)$ and $\mathcal{J}:=(J_0,J_1,\dots,J_t)$ such that $B^{\mathcal{I}}_{C^\perp}$ can be expressed in terms of $B^{\mathcal{J}}_{C}$. The following lemma shows that such choices of $\mathcal{I}$ and $\mathcal{J}$ exist.

\begin{lem}\label{lem:id-SWE}
    Let $C$ be a linear code of length $n$ over a finite commutative Frobenius ring~$R$, and let $\{a_0,a_1,\dots,a_t\}$ be the set of representatives described in Subsection~\ref{subsection:swe}. Let $\mathcal{I}:=(I_0,I_1,\dots,I_t)$ and $\mathcal{J}:=(J_0,J_1,\dots,J_t)$ be tuples of subsets of $\{0,1,\dots,t\}$, where for each $i\in\{0,1,\dots,t\}$,
    \begin{align*}
    I_i&:=\{j\in\{0,1,\dots,t\}\;\colon\;a_j a_k\neq 0\text{ for all }k\text{ such that }a_ia_k\neq 0\}\,;\\
    J_i&:=\{j\in\{0,1,\dots,t\}\;\colon\;a_k R\not\subseteq a_j R\text{ for all }k\text{ such that }a_ia_k\neq 0\}\,.
    \end{align*}
    Then for any multiset $X\subseteq t\cdot[n]$,
    \[B^{\mathcal{I}}_{C^\perp}(X)=\bigg(\frac{1}{|C|}\prod_{i=0}^t|a_iR|^{|M_i(X)|}\bigg)B^{\mathcal{J}}_C(X)\,.\]
\end{lem}
\begin{proof}
Consider the fact that for any $r\in R$ and $\cv\in R^n$,
\begin{equation}\label{eq:supp}
    \supp(r\cv)=\bigcup_{\substack{j\in\{0,1,\dots,t\}\colon\\r a_j\neq 0}}\hspace*{-2mm}\!\! S_j(\cv)\,.
\end{equation}
From (\ref{eq:supp}), we can deduce that the following conditions $\mathrm{(A1)}$ and $\mathrm{(A2)}$ are equivalent:
\begin{itemize}
    \item[$\mathrm{(A1)}$] For any $k\in\{0,1,\dots,t\}$, $\supp(a_k\cv)\subseteq \displaystyle\bigcup_{\substack{j\in\{0,1,\dots,t\}\colon\\a_j a_k\neq 0}}\hspace*{-2mm}\!\!  M_j(X)$\,;
    \item[$\mathrm{(A2)}$] For any $i\in\{0,1,\dots,t\}$, $S_i(\cv)\subseteq \displaystyle\bigcup_{j\in I_i}M_j(X)$\,.
\end{itemize}
Thus for any multiset $X\subseteq t\cdot[n]$,
\begin{equation}\label{eq:id-BCI}
    B^{\mathcal{I}}_{C^\perp}(X)=\Big|\Big\{\cv\in C^\perp\;\colon\;\supp(a_k\cv)\subseteq \displaystyle\bigcup_{\substack{j\in\{0,1,\dots,t\}\colon\\a_j a_k\neq 0}}\hspace*{-2mm}\!\!  M_j(X)\,\text{ for all }k=0,1,\dots,t\Big\}\Big|
\end{equation}
from (\ref{eq:BC}), where the code is $C^\perp$ and $\mathcal{F}=\mathcal{I}$.  

Now, we will show that the value in (\ref{eq:id-BCI}) can be obtained from the cardinality of the kernel of a particular module homomorphism. Consider the $R$-module $\bigoplus_{i=0}^t (a_i R)^{M_i(X)}$. Here, $(a_i R)^{M_i(X)}\cong(a_iR)^{|M_i(X)|}$ is the set of all $|M_i(X)|$-tuples over $a_i R$ indexed by the elements of $M_i(X)$. Note that $\bigoplus_{i=0}^t (a_i R)^{M_i(X)}$ is a submodule of $R^n$. Let $\mathbf{G}=(\gv_1\dots\gv_n)\in R^{m\times n}$ be a generator matrix of $C$. For each multiset $X\subseteq t\cdot[n]$, consider the module homomorphism 
\begin{eqnarray*}
  \sigma_X: & \displaystyle\bigoplus_{i=0}^t (a_i R)^{M_i(X)} &\rightarrow\;\;\quad R^{m}\\
          & (c_1,\dots,c_n)                                   &\mapsto\;\; \sum_{\ell=1}^n c_\ell\gv_\ell\,.
\end{eqnarray*}
Note that $\ker(\sigma_X)\subseteq C^\perp$ by definition. Moreover, $\cv\in \bigoplus_{i=0}^t (a_i R)^{M_i(X)}$ implies that for any $k\in\{0,1,\dots,t\}$ and any $j\in\{0,1,\dots,t\}$ such that $a_j a_k=0$, $a_kc_\ell=0$ for all $\ell\in M_j(X)$. Thus, (A1) holds whenever $\cv\in\ker(\sigma_X)$. On the other hand, suppose that (A1) holds for some $\cv\in C^\perp$. For any $j\in\{0,1,\dots,t\}$, note that if $a_j a_k=0$, then $a_k c_\ell=0$ for all $\ell\in M_j(X)$ by (A1). In other words, $(a_j R)^\perp\subseteq (c_\ell R)^\perp$ for all $\ell\in M_j(X)$, and this is equivalent to $c_\ell\in a_j R$ for all $\ell\in M_j(X)$ by Corollary~\ref{cor:dual}. We conclude that $\cv\in\bigoplus_{i=0}^t (a_i R)^{M_i(X)}$. Moreover, $\sigma_X(\cv)=0$ since $\cv\in C^\perp$. Therefore,
\begin{equation}\label{eq:Ker-sigma}
    \ker(\sigma_X)=\Big\{\cv\in C^\perp\;\colon\;\supp(a_k\cv)\subseteq\bigcup_{\substack{j\in\{0,1,\dots,t\}\colon\\a_j a_k\neq 0}}\hspace*{-2mm}\!\! M_j(X)\,\text{ for all }k=0,1,\dots,t\Big\}\,.
\end{equation}
The equation above implies that $B^{\mathcal{I}}_{C^\perp}(X)=|\ker(\sigma_X)|$ by (\ref{eq:id-BCI}).

Similarly, we will show that there exists a module homomorphism $\tau_X$ such that $B^{\mathcal{J}}_{C}(X)=|\ker(\tau_X)|$. From (\ref{eq:supp}), the following conditions $\mathrm{(B1)}$ and $\mathrm{(B2)}$ are equivalent: 
\begin{itemize}
    \item[$\mathrm{(B1)}$] For any $k\in\{0,1,\dots,t\}$, $\supp(a_k\cv)\subseteq\displaystyle\bigcup_{\substack{j\in\{0,1,\dots,t\}\colon\\a_k R\not\subseteq a_j R}} \hspace*{-2mm}\!\!  M_j(X)$\,;
    \item[$\mathrm{(B2)}$] For any $i\in\{0,1,\dots,t\}$, $S_i(\cv)\subseteq \displaystyle\bigcup_{j\in J_i} M_j(X)$\,.
\end{itemize}
Then for any multiset $X\subseteq t\cdot[n]$,
\begin{equation}\label{eq:id-BCJ}
    B^{\mathcal{J}}_{C}(X)=\Big|\Big\{\cv\in C\;\colon\;\supp(a_k\cv)\subseteq\displaystyle\bigcup_{\substack{j\in\{0,1,\dots,t\}\colon\\a_k R\not\subseteq a_j R}} \hspace*{-2mm}\!\!  M_j(X)\Big\}\Big|
\end{equation}
from (\ref{eq:BC}) for $\mathcal{F}=\mathcal{J}$. Consider another module homomorphism
\begin{eqnarray*}
  \tau_X: & C &\rightarrow\quad\;\; R^{n}\\
          & (c_1,\dots,c_n)                                   &\mapsto\;\; (\alpha_1 c_1,\dots,\alpha_nc_n)\,,
\end{eqnarray*}
where $\alpha_\ell=a_i$ if $\ell\in M_i(X)$. For any $\cv\in \ker(\tau_X)$ and any $j\in\{0,1,\dots,t\}$, note that $a_j c_\ell=\alpha_\ell c_\ell=0$ for all $\ell\in M_j(X)$. Now fix $k\in\{0,1,\dots,t\}$. For each $j\in\{0,1,\dots,t\}$ such that $a_k R\subseteq a_j R$, we can deduce that $a_k c_\ell=0$ for all $\ell\in M_j(X)$ since $a_j c_\ell=0$ for all $\ell\in M_j(X)$ and thus condition (B1) holds. On the other hand, if $\cv\in C$ satisfies the condition (B1), then fix any $k\in\{0,1,\dots, t\}$ and note that $\alpha_\ell c_\ell=a_k c_\ell=0$ for any $\ell\in M_k(X)$. Since $k$ is arbitrary, $\alpha_\ell c_\ell=0$ for all $\ell\in[n]$. Therefore,   
\begin{equation}\label{eq:Ker-tau}
    \ker(\tau_X)=\Big\{\cv\in C\;\colon\;\supp(a_k\cv)\subseteq\bigcup_{\substack{j\in\{0,1,\dots,t\}\colon\\a_k R\not\subseteq a_j R}}\hspace*{-2mm}\!\! M_j(X)\,\text{ for all }k=0,1,\dots,t\Big\}\,.
\end{equation}

From Corollary \ref{cor:span}, we deduce that $|\mathrm{im}(\sigma_X)|=|\mathrm{im}(\tau_X)|$ for any multiset $X\subseteq t\cdot[n]$. By applying the Module Isomorphism Theorem to $\sigma_X$ and $\tau_X$,
\[\frac{\displaystyle\prod_{i=0}^t|a_iR|^{|M_i(X)|}}{|\ker(\sigma_X)|}=\frac{\Big|\displaystyle\bigoplus_{i=0}^t (a_i R)^{M_i(X)}\Big|}{|\ker(\sigma_X)|}=|\mathrm{im}(\sigma_X)|=|\mathrm{im}(\tau_X)|=\frac{|C|}{|\ker(\tau_X)|}\,.\]
Apply Equations (\ref{eq:Ker-sigma}) and (\ref{eq:Ker-tau}) together with Equations (\ref{eq:id-BCI}) and (\ref{eq:id-BCJ}) to complete the proof. 
\end{proof}

For $\mathcal{F}=\mathcal{I}$, the matrix $\mathbf{P}^\mathcal{I}$ obtained from Theorem~\ref{thm:id-SWE} has a surprising and natural combinatorial interpretation as follows. 
\begin{prop}\label{prop:A}
Let $C$ be a linear code of length $n$ over a finite commutative Frobenius ring~$R$, and let $\mathcal{I}:=(I_0,I_1,\dots,I_t)$ be a tuple of subsets of $\{0,1,\dots,t\}$ as defined in Lemma~\ref{lem:id-SWE}. Then the matrix $\mathbf{A}:=\mathbf{P}^{\mathcal{I}}$ with respect to Theorem~\ref{thm:id-SWE} is the adjacency matrix of the poset of principal ideals of $R$ under set inclusion. Consequently, $\mathbf{A}$ is invertible. 
\end{prop}
\begin{proof}
Let $\mathbf{A}=(a_{ij})$. From Theorem~\ref{thm:id-SWE},
\[a_{ij}=p^{\mathcal{I}}_{ij}=\begin{cases}
    1,&\text{ if }j\in I_i\,;\\
    0,&\text{ otherwise.}
    \end{cases}\]
Note that the following statements are equivalent:
\begin{enumerate}[label=(\alph*)]
    \item $j\in I_i$\,;
    \item $a_j a_k\neq 0\text{ for all }k\text{ such that }a_ia_k\neq 0$;
    \item $a_i r=0$ for all $r\in R$ such that $a_j r=0$;
    \item $(a_jR)^\perp\subseteq (a_i R)^\perp$;
    \item $a_i R\subseteq a_j R$.
\end{enumerate}
The equivalences of statements (a), (b), (c) and (d) are straightforward to prove and the statements (d) and (e) are equivalent by Corollary~\ref{cor:dual}. The proposition follows.
\end{proof}

\begin{rem}\label{rem:A}
    Since $\mathbf{A}$ is an adjacency matrix, the entries of $\mathbf{A}^{-1}$ are the values of the M\"{o}bius function with respect to the poset. Moreover, there exists a permutation of indices $\{0,1,\dots,t\}$ such that $\mathbf{A}$ is an upper-triangular matrix. 
\end{rem}

Now we are ready to prove our main theorem.
\begin{theorem}\label{thm:main}
    Let $C$ be a linear code of length $n$ over a finite commutative Frobenius ring~$R$, and let $\{a_0,a_1,\dots,a_t\}$ be the set of representatives described in Subsection~\ref{subsection:swe}. Then 
    \[\swe_{C^\perp}(\xv)=\frac{1}{|C|}\swe_C\big(\mathbf{QDA}^{-1}\xv\big)\]
    for some $(t+1)\times (t+1)$ matrices $\mathbf{Q, D, A}$ whose rows and columns are indexed by $\{0,1,\dots,t\}$. In particular, $\mathbf{A}$ is the adjacency matrix of the poset of principal ideals of $R$ under set inclusion, $\mathbf{D}$ is the diagonal matrix with entries $|a_0 R|, |a_1 R|, \dots, |a_t R|$ respectively, and $\mathbf{Q}=~\!\!(q_{ij})$ is the symmetric matrix with 
    \begin{equation}\label{eq:Q}
        q_{ij}:=\begin{cases}
    0,&\text{ if }a_i\neq 0 \text{ and }a_k R\subseteq a_j R\text{ for some }k\text{ such that }a_ia_k\neq 0\,;\\
    1,&\text{ otherwise.}
    \end{cases}
    \end{equation}
    Equivalently,
    \begin{equation}\label{eq:Q-alt}
        q_{ij}=\begin{cases}
    1,&\text{ if }a_j R\subseteq (a_i R)^\perp\,;\\
    0,&\text{ otherwise.}
    \end{cases}
    \end{equation}
\end{theorem}
\begin{proof} 
Let $\mathcal{I}:=(I_0,I_1,\dots,I_t)$ and $\mathcal{J}:=(J_0,J_1,\dots,J_t)$ be tuples of subsets as defined in Lemma~\ref{lem:id-SWE}. Note that $I_i$ and $J_i$ are non-empty for any $i\in\{0,1,\dots,t\}$, so Theorem~\ref{thm:id-SWE} holds. Applying Theorem~\ref{thm:id-SWE} to $C^\perp$ with $\mathcal{F}=\mathcal{I}$ gives us
    \begin{equation}\label{eq:main-proof-1}
        \swe_{C^\perp}(\xv)=\sum_{Y\subseteq t\cdot[n]} B^{\mathcal{I}}_{C^\perp}(Y)\prod_{j=0}^t y_j^{|M_j(Y)|}\,,
    \end{equation}
    where $\xv=\mathbf{P}^{\mathcal{I}}\yv$. Here, $\mathbf{A}:=\mathbf{P}^{\mathcal{I}}$ is the adjacency matrix of the poset of principal ideals of $R$ under set inclusion; see Proposition~\ref{prop:A}. Since $\mathbf{A}$ is invertible, $\yv=\mathbf{A}^{-1}\xv$. By Lemma~\ref{lem:id-SWE},
    \begin{align*}
        \sum_{Y\subseteq t\cdot[n]} B^{\mathcal{I}}_{C^\perp}(Y)\prod_{j=0}^t y_j^{|M_j(Y)|}
        &=\sum_{Y\subseteq t\cdot[n]} \frac{\prod_{i=0}^t|a_iR|^{|M_i(Y)|}}{|C|}B^{\mathcal{J}}_C(Y)\prod_{j=0}^t y_j^{|M_j(Y)|}\\
        &=\frac{1}{|C|}\sum_{Y\subseteq t\cdot[n]}B^{\mathcal{J}}_C(Y)\prod_{j=0}^t (|a_j R|y_j)^{|M_j(Y)|}\\
        &=\frac{1}{|C|}\sum_{Y\subseteq t\cdot[n]}B^{\mathcal{J}}_C(Y)\prod_{j=0}^t w_j^{|M_j(Y)|},
    \end{align*}
    where $\wv=\mathbf{D}\yv=\mathbf{DA}^{-1}\xv$. Combine the above equation with (\ref{eq:main-proof-1}) and then apply Theorem~\ref{thm:id-SWE} one more time with $\mathcal{F}=\mathcal{J}$ to get
    \[\swe_{C^\perp}(\xv)=\frac{1}{|C|}\swe_C(\mathbf{Q}\wv),\]
    where $\mathbf{Q}:=\mathbf{P}^\mathcal{J}$ is the matrix in Theorem~\ref{thm:id-SWE} when $\mathcal{F}=\mathcal{J}$. Here, $\mathbf{Q}\wv=\mathbf{QDA}^{-1}\xv$, as desired. Note that (\ref{eq:Q}) is equivalent to the definition of $\mathbf{P}^{\mathcal{J}}$ in Theorem~\ref{thm:id-SWE}. To prove the alternative form of $\mathbf{Q}$ in (\ref{eq:Q-alt}), note that the following statements are equivalent:
    \begin{enumerate}[label=(\alph*)]
    \item $j\in J_i$\,;
    \item $a_k R\not\subseteq a_j R\text{ for all }k\text{ such that }a_ia_k\neq 0$;
    \item $a_i a_k=0$ for all $k$ such that $a_k R\subseteq a_j R$;
    \item $a_j R\subseteq (a_i R)^\perp$.
\end{enumerate}
It is not hard to show that these four statements are equivalent. The fact that $\mathbf{Q}$ is symmetric follows straightforwardly from (\ref{eq:Q}) or (\ref{eq:Q-alt}). 
\end{proof}

The expression in Theorem~\ref{thm:main} reveals a surprising implication, namely that the matrix appearing in Wood's identity \cite{Wood1999} in terms of generating character (Theorem~\ref{thm:MW-Wood}) is an integer matrix that can be decomposed into a product of a zero-one symmetric matrix, a diagonal matrix, and the inverse of a zero-one matrix. Moreover, the latter matrix is upper triangular under a suitable permutation of indices; see Remark \ref{rem:A}.
\begin{cor}\label{cor:GC-id}
    Let $R$ be a finite commutative Frobenius ring, and let $\chi:R\rightarrow\mathbb{C}$ be a generating character of $R$. If $\mathbf{S}=(s_{ij})$ is a $(t+1)\times (t+1)$ matrix whose rows and columns are indexed by $\{0,1,\dots,t\}$, where $s_{ij}:=\sum_{r\,\approx\,a_j}\chi(ra_i),$
    then
    \[\mathbf{S=QDA}^{-1},\]
    where $\mathbf{Q, D, A}$ are the matrices described in Theorem~\ref{thm:main}.
\end{cor}
\begin{rem}
    We can apply Theorem~\ref{thm:MW-Wood} to $C$ and $C^\perp$ respectively and deduce from Lemmas~\ref{lem:Wood} and \ref{lem:AA} that $\mathbf{S}^2=|R|\;\mathbf{I}$, where $\mathbf{I}$ is the identity matrix of size $(t+1)\times (t+1)$. Thus, the matrix $\mathbf{Q}$ must be invertible by Corollary~\ref{cor:GC-id}.  
\end{rem}
\noindent
We also note that our MacWilliams identity in Theorem~\ref{thm:main} contains combinatorial information that is not present in Wood's MacWilliams identity in Theorem~\ref{thm:MW-Wood}. The matrix~$\mathbf{A}$ records the structure of the poset of the principal ideals under set inclusion; see Proposition~\ref{prop:A}. As for the matrix $\mathbf{Q}$, Equation (\ref{eq:Q-alt}) shows that this matrix depends on the dual of each principal ideal.
It might be worth noting that the poset of principal ideals also features in many important results in coding theory; see \cite{ByGrHo08,GrHoFaWoZu14, GrSc2004} for examples. 
\begin{rem}
    Suppose that a ring decomposition of $R$ is known: $R=R_1\oplus\cdots\oplus R_s$. We remark that the matrices $\mathbf{Q}, \mathbf{D}, \mathbf{A}$ for $R$ can be obtained from the analogous matrices $\mathbf{Q}_i, \mathbf{D}_i, \mathbf{A}_i$ corresponding to $R_i$ for $i=1,\dots,s$. In particular, from the definition of the matrices and the fact that the ideals of $R$ are direct sums of ideals of $R_1,\dots, R_s$, it is not hard to see that, with a suitable ordering of the representatives $\{a_0,a_1,\dots,a_t\}$,
    \begin{align*}
        \mathbf{Q} &= \mathbf{Q}_1\otimes\cdots\otimes\mathbf{Q}_s,\\
        \mathbf{D} &= \mathbf{D}_1\otimes\cdots\otimes\mathbf{D}_s,\\
        \mathbf{A} &= \mathbf{A}_1\otimes\cdots\otimes\mathbf{A}_s,
    \end{align*}
    where $\otimes$ denotes the Kronecker product.
\end{rem}

\section{Special cases and applications}\label{sec:applications}
In this section, we will consider special classes of rings and determine more explicit forms of $\mathbf{Q, D, A}$ from Theorem~\ref{thm:main}. We will also compare Theorem~\ref{thm:main} with known results obtained from calculations using the generating character (Theorem~\ref{thm:MW-Wood}). Moreover, we apply our results to the symmetrized Lee weight enumerator and the Lee weight enumerator, and compare with the known results originally obtained from Gray maps in \cite{LiuLiu15}. Finally, we use Theorem~\ref{thm:main} to obtain a new proof of the MacWilliams identity for the Hamming weight enumerator.
\subsection{Finite chain rings}\label{subsection:FCR}
Throughout this subsection, let $R$ be a finite commutative chain ring, that is, a ring whose lattice of ideals forms a chain. Following \cite{NoSa}, we denote the generator of the maximal ideal of $R$ by $\gamma$ and denote by $\nu$ the nilpotency index of $\gamma$, namely the smallest positive integer such that $\gamma^\nu=0$. Moreover, we denote the cardinality of the residue field $R/\gamma R$ by $q$. It is easy to order the ideals and the equivalence classes of $R$ with respect to $\approx$ defined in Subsection~\ref{subsection:swe} since $R$ is a principal ideal ring with $\nu+1$ ideals
\[\gamma^\nu R\subsetneq \cdots\subsetneq \gamma R\subsetneq R\]
with cardinalities $1, q,\dots,q^\nu$ respectively \cite[Lemma 2.4]{NoSa} and $\nu+1$ equivalence classes
\[\gamma^\nu R^\times, \dots, \gamma R^\times, R^\times.\]
We may therefore choose $a_0=\gamma^\nu,\dots, a_{\nu-1}=\gamma, a_\nu=1$ as the representatives of the equivalence classes. This choice implies that $a_i R\subseteq a_j R$  if and only if $i\leq j$. Therefore,
\[\mathbf{A}=\begin{pmatrix}
    1 & 1 &\cdots &1\\
    0 & 1 &\cdots &1\\
    \vdots &\vdots &\ddots &\vdots\\
    0 &0 &\cdots &1
\end{pmatrix}, \quad
\mathbf{D}=\begin{pmatrix}
    1 & 0 &\cdots &0\\
    0 & q &\cdots &0\\
    \vdots &\vdots &\ddots &\vdots\\
    0 &0 &\cdots &q^\nu
\end{pmatrix}, \quad
\mathbf{Q}=\begin{pmatrix}
    1 &\cdots &1 &1\\
    1 &\cdots &1 &0\\
    \vdots &\iddots &\vdots &\vdots\\
    1 &\cdots &0 &0
\end{pmatrix}.\]
Notice that
\begin{equation}\label{eq:QAJ}
    \mathbf{Q=AJ},
\end{equation}
where $\mathbf{J}$ is the anti-diagonal matrix of size $(\nu+1)\times(\nu+1)$ with all 1-entries. If $\nu=1$, then~$R$ is the finite field of $q$ elements and
\[\mathbf{QDA}^{-1}=\begin{pmatrix}
    1 & q-1\\
    1 & -1
\end{pmatrix}.\]
Then Theorem~\ref{thm:main} becomes the well-known original MacWilliams identity \cite{macwilliams63}. For $\nu\geq 2$, the explicit form of $\mathbf{QDA}^{-1}$ is
\[\mathbf{QDA}^{-1}=\begin{pmatrix}
    1 & q-1 & q^2-q &\cdots &q^{\nu}-q^{\nu-1} \\
    1 & q-1 & q^2-q &\cdots &-q^{\nu-1}\\
    \vdots &\vdots &\vdots &\iddots  &\vdots\\
    1 & q-1 &-q &\cdots &0\\
    1 &-1 &0 &\cdots &0
\end{pmatrix}.\]
This means that we can obtain the following identity involving the generating character from Corollary \ref{cor:GC-id} and the explicit form of MacWilliams identity from Theorem~\ref{thm:main}.
\begin{cor}\label{cor:FCR}
    Let $C$ be a linear code over a finite commutative chain ring $R$, and let $\chi:R\rightarrow \mathbb{C}$ be a generating character of $R$. 
    Then
    \[
    \swe_{C^\perp}(\xv)=\frac{1}{|C|}\swe_C(\mathbf{S}\xv)\,,
    \]
    where $\mathbf{S}=(s_{ij})$ is a $(\nu+1)\times (\nu+1)$ matrix whose rows and columns are indexed by $\{0,1,\dots,\nu\}$ and for any $i,j\in\{0,1,\dots,\nu\}$,
    \[s_{ij}:=\sum_{r\in \gamma^{\nu-j}R^\times}\chi(r \gamma^{\nu-i})=
    \begin{cases}
        1 & \text{if }j=0\,;\\
        q^{j}-q^{j-1} & \text{if }1\leq j\leq \nu-i\,;\\
        -q^{\nu-i} & \text{if }j= \nu+1-i\,;\\
        0 & \text{otherwise.}
    \end{cases}\]
\end{cor}
\begin{ex}\label{ex:Z4/<x^3-2,2x>}
Consider the local ring \[R:=\ZZ_4[x]/\langle x^3-2,2x\rangle=\{a+bx+cx^2\;\colon\;a\in\ZZ_4,\; b,c\in\ZZ_2,\; x^3=2,\; 2x=0\}\,,\] 
where $\ZZ_4:=\{0,1,2,3\}$ and $\ZZ_2:=\{0,1\}$ are the rings of integers modulo 4 and 2, respectively. Here, $q=2$, $\nu=4$ and we choose $\gamma=x$ as the generator of the maximal ideal of $R$. 
The matrix $\mathbf{S}$ according to Corollary \ref{cor:FCR} is 
\[\mathbf{S}=\begin{pmatrix}
    1 & q-1 & q^2-q &q^3-q^2 &q^4-q^3 \\
    1 & q-1 & q^2-q &q^3-q^2 &-q^3\\
    1 & q-1 & q^2-q &-q^2 &0\\
    1 & q-1 &-q &0 &0\\
    1 &-1 &0 &0 &0
\end{pmatrix}=\begin{pmatrix*}[r]
    1 & 1 & 2 &4 &8 \\
    1 & 1 & 2 &4 &-8\\
    1 & 1 & 2 &-4 &0\\
    1 & 1 &-2 &0 &0\\
    1 &-1 &0 &0 &0
\end{pmatrix*}.\]
This matrix represents the same transformation as $S_3$ in \cite[Table 2]{DoSaSz19} obtained from the generating character, the only difference being that the representatives are in a different order. In particular, $a_0=0$, $a_1=1$, $a_2=\gamma$, $a_3=\gamma^2$, $a_4=\gamma^3$ in \cite[Table 2]{DoSaSz19}, while our choice of representatives is $a_0=0$, $a_1=\gamma^3$, $a_2=\gamma^2$, $a_3=\gamma$, $a_4=1$.
\end{ex}

\subsection{Finite principal ideal rings}
Throughout this subsection, let $R$ be a finite commutative principal ideal ring, namely the ring in which any ideal $I$ of $R$ satisfies $I=aR$ for some $a\in R$. This class of rings includes integer modulo rings and finite chain rings. In fact, a finite commutative principal ideal ring is isomorphic to a product of finite chain rings; see \cite[Proposition~2.7]{DoKiKu09}.

Observe that for each $i\in\{0,1,\dots,t\}$, $(a_i R)^\perp=a_j R$ for some $j\in\{0,1,\dots,t\}$. From Lemma~\ref{lem:AA}, $(a_j R)^\perp=a_i R$. This implies that
\begin{equation}\label{eq:PIR}
    \{a_0 R,a_1 R,\dots,a_t R\}=\{(a_0 R)^\perp,(a_1 R)^\perp,\dots,(a_t R)^\perp\}\,.
\end{equation}
Therefore, the latter set is comprised of all $t+1$ principal ideals of $R$.
\begin{cor}\label{cor:FPIR}
    Let $R$ be a finite commutative principal ideal ring, and let $\mathbf{A}=(a_{ij})$ and $\mathbf{Q}=(q_{ij})$ be the matrices defined in Theorem~\ref{thm:main}. Then
    \[\mathbf{Q}=\mathbf{AP}\]
    for some permutation matrix $\mathbf{P}$. Furthermore, there exists a choice of representatives of equivalence classes $\{a_0, a_1,\dots,a_t\}$ such that (\ref{eq:QAJ}) holds, namely that $\mathbf{P}=\mathbf{J}$ is the anti-diagonal matrix with all 1-entries of size $(t+1)\times (t+1)$.
\end{cor}
\begin{proof}
From (\ref{eq:PIR}), we know that there exists an involution $\phi$ on $\{0,1,\dots,t\}$ such that $(a_j R)^\perp=a_{\phi(j)}R$ for all $j=0,1,\dots,t$. From (\ref{eq:Q-alt}) and Corollary~\ref{cor:dual}, $q_{ij}=1$ if and only if $a_j R\subseteq (a_i R)^\perp$ if and only if $a_i R\subseteq (a_j R)^\perp=a_{\phi(j)} R$ if and only if the matrix entry $a_{i\phi(j)}$ equals $1$. Thus, $\mathbf{Q}=\mathbf{AP}$, where $\mathbf{P}$ is a permutation matrix with respect to the permutation~$\phi^{-1}=\phi$. 

Now suppose that $R$ is isomorphic to a product of $s$ finite chain rings $R_1,\dots, R_s$ with respect to the decomposition in \cite[Proposition 2.7]{DoKiKu09}. For each $m\in\{1,\dots,s\}$, let $\gamma_m$ and $\nu_m$ be the generator of the maximal ideal and the nilpotency index of $R_m$, respectively. If there exists an index $i\in\{0,1,\dots,t\}$ such that $(a_i R)^\perp=a_i R$, namely that $a_i R$ is a self-dual code of length 1, then we know from \cite[Theorem 6.4]{DoKiKu09} that there exist self-dual linear codes $C_1,\dots,C_s$ of length 1 over $R_1,\dots, R_s$ respectively. Recall from the previous subsection that for each $m\in\{1,\dots,s\}$, $R_m$ has exactly $\nu_m+1$ distinct ideals
\[\gamma_m^{\nu_m} R_m\subsetneq \cdots\subsetneq \gamma_m R_m\subsetneq R_m\,,\]
so $C_m=\gamma_m^{j} R_m$ for some  $j\in\{0,1,\dots,\nu_m\}$. By Lemma~\ref{lem:Wood}, there can be at most one $C_m$ such that $C_m=C_m^\perp$. Therefore, there is at most one possible index $i\in\{0,1,\dots,t\}$ such that $(a_i R)^\perp=a_i R$. This means that we can choose a suitable order of representatives such that 
$(a_i R)^\perp=a_{t-i} R$ for any $i\in\{0,1,\dots,t\}$. In this case, the corresponding permutation matrix is an anti-diagonal matrix and thus (\ref{eq:QAJ}) holds.
\end{proof}
\begin{rem}
If (\ref{eq:QAJ}) holds, then
\[\mathbf{S}=\mathbf{A(JD)A}^{-1}\]
from Corollary \ref{cor:GC-id}. Here, $\mathbf{J}\mathbf{D}$ is the anti-diagonal matrix of size $(t+1)\times(t+1)$ with entries $|a_0 R|, |a_1 R|,\dots,|a_t R|$ read from the bottom left.
\end{rem}
For the case when $R$ is the ring of integers modulo $m$, we can choose the representatives $\{a_0,a_1\dots,a_t\}$ of the equivalences classes such that
\[1=|a_0 R|<|a_1 R|<\cdots<|a_t R|=m\,.\]
Here, $(a_i R)^\perp =a_{t-i} R$ for all $i=0,1,\dots,t$ and thus (\ref{eq:QAJ}) holds, as stated in Corollary~\ref{cor:FPIR}. 
\begin{ex}\label{ex:Z12}
Consider the case when $R=\ZZ_{12}:=\{0,1,\dots,11\}$ is the (non-local) ring of integers modulo 12 and choose $a_0=0$, $a_1=6$, $a_2=4$, $a_3=3$, $a_4=2$, $a_5=1$. The lattice of principal ideals of $\ZZ_{12}$ under set inclusion is presented in Figure~\ref{fig:Z12} below.
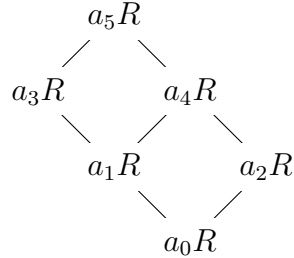
\begin{figure}[ht]
    \centering
    \begin{tikzpicture}
        \draw (0,0) -- (1,1) -- (0,2) -- (-1,1);
        \draw (0,0) -- (-1,1) -- (-2,2) -- (-1,3) -- (0,2);
        \draw (0,0) node [fill=white] {$a_0 R$};
        \draw (-1,1) node [fill=white] {$a_1 R$};
        \draw (1,1) node [fill=white] {$a_2 R$};
        \draw (0,2) node [fill=white] {$a_4 R$};
        \draw (-2,2) node [fill=white] {$a_3 R$};
        \draw (-1,3) node [fill=white] {$a_5 R$};
    \end{tikzpicture}
    \caption{Hasse diagram for principal ideals of $\ZZ_{12}$.}
    \label{fig:Z12}
\end{figure}

The matrices $\mathbf{A, D, Q}$ according to Theorem~\ref{thm:main} are
    \[\mathbf{A}=\begin{pmatrix}
    1 & 1 & 1 & 1 & 1 &1\\
    0 & 1 & 0 & 1 & 1 &1\\
    0 & 0 & 1 & 0 & 1 &1\\
    0 & 0 & 0 & 1 & 0 &1\\
    0 & 0 & 0 & 0 & 1 &1\\
    0 & 0 & 0 & 0 & 0 &1
\end{pmatrix},\quad
\mathbf{D}=\begin{pmatrix}
    1 & 0 & 0 & 0 & 0 &0\\
    0 & 2 & 0 & 0 & 0 &0\\
    0 & 0 & 3 & 0 & 0 &0\\
    0 & 0 & 0 & 4 & 0 &0\\
    0 & 0 & 0 & 0 & 6 &0\\
    0 & 0 & 0 & 0 & 0 &12
\end{pmatrix},\quad
\mathbf{Q}=\begin{pmatrix}
    1 & 1 & 1 & 1 & 1 &1\\
    1 & 1 & 1 & 0 & 1 &0\\
    1 & 1 & 0 & 1 & 0 &0\\
    1 & 0 & 1 & 0 & 0 &0\\
    1 & 1 & 0 & 0 & 0 &0\\
    1 & 0 & 0 & 0 & 0 &0
\end{pmatrix}.\]
Note that (\ref{eq:QAJ}) holds. The transformation matrix with respect to Theorem~\ref{thm:main} is
\[\mathbf{QDA}^{-1}=
\begin{pmatrix*}[r]
    1 & 1 & 2 & 2 & 2 &4\\
    1 & 1 & 2 & -2 & 2 &-4\\
    1 & 1 & -1 & 2 & -1 &-2\\
    1 & -1 & 2 & 0 & -2 &0\\
    1 & 1 & -1 & -2 & -1 &2\\
    1 & -1 & -1 & 0 & 1 &0
\end{pmatrix*}.\]
    It is easy to check that this matrix matches the transformation matrix $\mathbf{S}$ in Theorem~\ref{thm:MW-Wood} obtained from the generating character $\chi$, where $\chi(a):=e^{\pi i a/6}$ for $a\in\{0,1,\dots,11\}$.
\end{ex}
For the next example, we demonstrate how the MacWilliams identity in Theorem~\ref{thm:main} contains information regarding the combinatorial structure of the principal ideals of the ring. In particular, we give another example which gives a different transformation matrix~$\mathbf{S}$ compared to Example~\ref{ex:Z12}, but with the same matrices~$\mathbf{A}$ and~$\mathbf{Q}$.
\begin{ex}\label{ex:F2+uF2+vF2}
Consider the (non-local) principal ideal ring 
\begin{equation*}
    R:=\FF_2+u\FF_2+v\FF_2=\{a+bu+cv\;\colon\;a,b,c\in\FF_2,u^2=0,v^2=v,uv=vu=0\}
\end{equation*}
which was studied in \cite{LiuLiu15}. It is easy to check that $R$ has six equivalence classes with respect to the equivalence relation $\approx$ and we choose $a_0=0$, $a_1=u$, $a_2=v$, $a_3=1+v$, $a_4=u+v$, $a_5=1$. The six principal ideals of $R$ are
    \begin{align*}
        a_0 R &=\{0\}\,; &a_2 R &=\{0,v\}\,; &a_4 R&=\{0,u,v,u+v\}\,;\\
        a_1 R &=\{0,u\}\,; &a_3 R &=\{0,u,1+v,1+u+v\}\,; &a_5 R&=R\,.
    \end{align*}
    We know that $R$ is not a chain ring from its lattice of principal ideals under set inclusion as shown in Figure \ref{fig:F2+uF2+vF2} below. \begin{figure}[ht]
    \centering
    \begin{tikzpicture}
        \draw (0,0) -- (1,1) -- (0,2) -- (-1,1);
        \draw (0,0) -- (-1,1) -- (-2,2) -- (-1,3) -- (0,2);
        \draw (0,0) node [fill=white] {$a_0 R$};
        \draw (-1,1) node [fill=white] {$a_1 R$};
        \draw (1,1) node [fill=white] {$a_2 R$};
        \draw (0,2) node [fill=white] {$a_4 R$};
        \draw (-2,2) node [fill=white] {$a_3 R$};
        \draw (-1,3) node [fill=white] {$a_5 R$};
    \end{tikzpicture}
    \caption{Hasse diagram for principal ideals of $\FF_2+u\FF_2+v\FF_2$.}
    \label{fig:F2+uF2+vF2}
\end{figure}
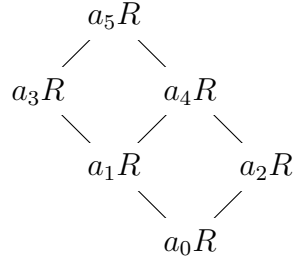

The matrices $\mathbf{A, D, Q}$ according to Theorem~\ref{thm:main} are
    \[\mathbf{A}=\begin{pmatrix}
    1 & 1 & 1 & 1 & 1 &1\\
    0 & 1 & 0 & 1 & 1 &1\\
    0 & 0 & 1 & 0 & 1 &1\\
    0 & 0 & 0 & 1 & 0 &1\\
    0 & 0 & 0 & 0 & 1 &1\\
    0 & 0 & 0 & 0 & 0 &1
\end{pmatrix},
\quad
\mathbf{D}=\begin{pmatrix}
    1 & 0 & 0 & 0 & 0 &0\\
    0 & 2 & 0 & 0 & 0 &0\\
    0 & 0 & 2 & 0 & 0 &0\\
    0 & 0 & 0 & 4 & 0 &0\\
    0 & 0 & 0 & 0 & 4 &0\\
    0 & 0 & 0 & 0 & 0 &8
\end{pmatrix},
\quad
\mathbf{Q}=\begin{pmatrix}
    1 & 1 & 1 & 1 & 1 &1\\
    1 & 1 & 1 & 0 & 1 &0\\
    1 & 1 & 0 & 1 & 0 &0\\
    1 & 0 & 1 & 0 & 0 &0\\
    1 & 1 & 0 & 0 & 0 &0\\
    1 & 0 & 0 & 0 & 0 &0
\end{pmatrix}.\]
Then the transformation matrix with respect to Theorem~\ref{thm:main} is
\begin{equation}\label{eq:F2+uF2+vF2}
    \mathbf{QDA}^{-1}=
\begin{pmatrix*}[r]
    1 & 1 & 1 & 2 & 1 &2\\
    1 & 1 & 1 & -2 & 1 &-2\\
    1 & 1 & -1 & 2 & -1 &-2\\
    1 & -1 & 1 & 0 & -1 &0\\
    1 & 1 & -1 & -2 & -1 &2\\
    1 & -1 & -1 & 0 & 1 &0
\end{pmatrix*}.
\end{equation}
Now we compare this example with Example~\ref{ex:Z12}. Observe that $\FF_2+u\FF_2+v\FF_2$ is not isomorphic to $\ZZ_{12}$, but they produce the same matrices $\mathbf{A}$ and $\mathbf{Q}$. The reasons for this are as follows: as shown in Figures~\ref{fig:Z12} and \ref{fig:F2+uF2+vF2}, both rings have the same poset of principal ideals, so they share the adjacency matrix~$\mathbf{A}$. Moreover, both rings satisfy $(a_i R)^\perp=a_{t-i} R$ for any $i\in\{0,1,\dots,t\}$ and hence they share the same matrix~$\mathbf{Q}$.
\end{ex}
\subsection{Miscellaneous rings and application to other enumerators}
Dougherty, Salt\"{u}rk, and Szabo \cite{DoSaSz19} gave the generating characters for all commutative local Frobenius rings of 16 elements and presented their transformation matrix $\mathbf{S}$ with respect to the definition in Theorem~\ref{thm:MW-Wood}. We can also easily find these transformation matrices with Theorem~\ref{thm:main}. We provide an example below.
\begin{ex}\label{ex:F2[u,v]} Consider the local ring
    \[R:=\FF_2[u,v]/\langle u^2,v^2\rangle=\{a+bu+cv+duv\;\colon\;a,b,c,d\in\FF_2,\; u^2=0,\; v^2=0\}\,.\]
    Following the order of equivalence classes in \cite[Example 1]{DoSaSz19}, we choose $a_0=0$, $a_1=1$, $a_2=u$, $a_3=v$, $a_4=u+v$, $a_5=uv$. The six principal ideals of $R$ are
    \begin{align*}
        a_0 R &=\{0\}\,; &a_2 R &=\{0,u,uv,u+uv\}\,; &a_4 R&=\{0,u+v,uv,u+v+uv\}\,;\\
        a_1 R &=R\,; &a_3 R &=\{0,v,uv,v+uv\}\,; &a_5 R&=\{0,uv\}\,.
    \end{align*}
    However, $R$ is not a principal ideal ring. The lattice of principal ideals under set inclusion is shown in Figure \ref{fig:F2+uF2+vF2+uvF2} below. 

    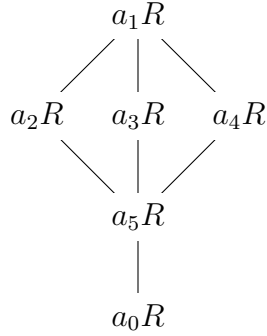
\begin{figure}[ht]
    \centering
    \begin{tikzpicture}
        \draw (0,0) -- (0,1.33) -- (0,2.67) -- (0,4);
        \draw (0,1.33) -- (-1.33,2.67) -- (0,4) -- (1.33,2.67) -- (0,1.33);
        \draw (0,0) node [fill=white] {$a_0 R$};
        \draw (0,4) node [fill=white] {$a_1 R$};
        \draw (-1.33,2.67) node [fill=white] {$a_2 R$};
        \draw (0,2.67) node [fill=white] {$a_3 R$};
        \draw (1.33,2.67) node [fill=white] {$a_4 R$};
        \draw (0,1.33) node [fill=white] {$a_5 R$};
    \end{tikzpicture}
    \caption{Hasse diagram for principal ideals of $\FF_2[u,v]/\langle u^2,v^2\rangle$.}
    \label{fig:F2+uF2+vF2+uvF2}
\end{figure}

    The matrices $\mathbf{A, D, Q}$ according to Theorem~\ref{thm:main} are
    \[\mathbf{A}=\begin{pmatrix}
    1 & 1 & 1 & 1 & 1 &1\\
    0 & 1 & 0 & 0 & 0 &0\\
    0 & 1 & 1 & 0 & 0 &0\\
    0 & 1 & 0 & 1 & 0 &0\\
    0 & 1 & 0 & 0 & 1 &0\\
    0 & 1 & 1 & 1 & 1 &1
\end{pmatrix},\quad
\mathbf{D}=\begin{pmatrix}
    1 & 0 & 0 & 0 & 0 &0\\
    0 & 16 & 0 & 0 & 0 &0\\
    0 & 0 & 4 & 0 & 0 &0\\
    0 & 0 & 0 & 4 & 0 &0\\
    0 & 0 & 0 & 0 & 4 &0\\
    0 & 0 & 0 & 0 & 0 &2
\end{pmatrix},\quad
\mathbf{Q}=\begin{pmatrix}
    1 & 1 & 1 & 1 & 1 &1\\
    1 & 0 & 0 & 0 & 0 &0\\
    1 & 0 & 1 & 0 & 0 &1\\
    1 & 0 & 0 & 1 & 0 &1\\
    1 & 0 & 0 & 0 & 1 &1\\
    1 & 0 & 1 & 1 & 1 &1
\end{pmatrix}.\]
Then the transformation matrix with respect to Theorem~\ref{thm:main} is
\[\mathbf{QDA}^{-1}=
\begin{pmatrix*}[r]
    1 & 8 & 2 & 2 & 2 &1\\
    1 & 0 & 0 & 0 & 0 &-1\\
    1 & 0 & 2 & -2 & -2 &1\\
    1 & 0 & -2 & 2 & -2 &1\\
    1 & 0 & -2 & -2 & 2 &1\\
    1 & -8 & 2 & 2 & 2 &1
\end{pmatrix*}.\]
This is the same matrix as $S$ in \cite[Example 1]{DoSaSz19} which was obtained from the generating character. This example also shows that $\mathbf{Q}$ is not necessarily a product of $\mathbf{A}$ and a permutation matrix as in the case when the ring is a finite principal ideal ring; see Corollary \ref{cor:FPIR}.
\end{ex}

For our next example, we apply Theorem~\ref{thm:main} to reprove the MacWilliams identities for the Lee weight enumerator and the symmetrized Lee weight enumerator \cite[Theorem 3.8]{LiuLiu15}, which were originally obtained using a Gray map. 
\begin{ex}\label{ex:Lee}
    Let $C$ be a linear code over $R:=\FF_2+u\FF_2+v\FF_2$ as defined in Example~\ref{ex:F2+uF2+vF2}. Let $\mathrm{Lee}_C(x,y)$ denote the \emph{Lee weight enumerator} of $C$ following \cite[Definition 3.1]{LiuLiu15}, and let $\mathrm{slwe}_C(x,y)$ denote the \emph{symmetrized Lee weight enumerator} of $C$ following the definition of $\mathrm{Swe}_C(x,y)$ in \cite[Definition 3.4]{LiuLiu15}. We remark that the authors in \cite{LiuLiu15} use the term ``symmetrized weight enumerator", but we use a different term since our definition of symmetrized weight enumerator is not the same.

    From the definition of Lee weight enumerator in \cite[Definition 3.1]{LiuLiu15}, it can be shown that 
\[\mathrm{Lee}_C(x,y)=\mathrm{swe}_C(x^3,xy^2,y^3,xy^2,x^2y,x^2y)\,.\]
From Theorem~\ref{thm:main} and (\ref{eq:F2+uF2+vF2}), $\mathrm{Lee}_{C^\perp}(x,y)=\frac{1}{|C|}\swe_C(\yv)$, where
\[
\begin{pmatrix}
    y_0\\y_1\\y_2\\y_3\\y_4\\y_5
\end{pmatrix}=
\begin{pmatrix*}[r]
    1 & 1 & 1 & 2 & 1 &2\\
    1 & 1 & 1 & -2 & 1 &-2\\
    1 & 1 & -1 & 2 & -1 &-2\\
    1 & -1 & 1 & 0 & -1 &0\\
    1 & 1 & -1 & -2 & -1 &2\\
    1 & -1 & -1 & 0 & 1 &0
\end{pmatrix*}
\begin{pmatrix}
    x^3\\xy^2\\y^3\\xy^2\\x^2y\\x^2y
\end{pmatrix}
=
\begin{pmatrix}
    x^3+3x^2y+3xy^2+y^3\\x^3-x^2y-xy^2+y^3\\x^3-3x^2y+3xy^2-y^3\\x^3-x^2y-xy^2+y^3\\x^3+x^2y-xy^2-y^3\\x^3+x^2y-xy^2-y^3
\end{pmatrix}.
\]
This gives us the MacWilliams identity for the Lee weight enumerator \cite[Theorem 3.8(i)]{LiuLiu15}:
\[\mathrm{Lee}_{C^\perp}(x,y)=\frac{1}{|C|}\mathrm{Lee}_C(x+y,x-y)\,.\]
From the definition of the symmetrized Lee weight enumerator in \cite[Definition 3.4]{LiuLiu15}, one can see that
\[\mathrm{slwe}_C(x_0,x_1,x_2,x_3)=\mathrm{swe}_C(x_0,x_2,x_3,x_2,x_1,x_1)\,.\]
From Theorem~\ref{thm:main} and (\ref{eq:F2+uF2+vF2}), $\mathrm{slwe}_{C^\perp}(x_0,x_1,x_2,x_3)=\frac{1}{|C|}\swe_C(\yv)$, where
\[
\begin{pmatrix}
    y_0\\y_1\\y_2\\y_3\\y_4\\y_5
\end{pmatrix}=
\begin{pmatrix*}[r]
    1 & 1 & 1 & 2 & 1 &2\\
    1 & 1 & 1 & -2 & 1 &-2\\
    1 & 1 & -1 & 2 & -1 &-2\\
    1 & -1 & 1 & 0 & -1 &0\\
    1 & 1 & -1 & -2 & -1 &2\\
    1 & -1 & -1 & 0 & 1 &0
\end{pmatrix*}
\begin{pmatrix}
    x_0\\x_2\\x_3\\x_2\\x_1\\x_1
\end{pmatrix}
=
\begin{pmatrix}
    x_0+3x_1+3x_2+x_3\\x_0-x_1-x_2+x_3\\x_0-3x_1+3x_2-x_3\\x_0-x_1-x_2+x_3\\x_0+x_1-x_2-x_3\\x_0+x_1-x_2-x_3
\end{pmatrix}.
\]
Therefore, we obtain the following MacWilliams identity for the symmetrized Lee weight enumerator \cite[Theorem 3.8(ii)]{LiuLiu15}: 
\begin{align*}
    &\mathrm{slwe}_{C^\perp}(x_0,x_1,x_2,x_3)\\
    &\!\!=\!\frac{1}{|C|}\mathrm{slwe}_C(x_0+3x_1+3x_2+x_3,x_0+x_1-x_2-x_3,x_0-x_1-x_2+x_3,x_0-3x_1+3x_2-x_3)\,.
\end{align*}
\end{ex}

Finally, we demonstrate how Theorem~\ref{thm:main} leads to a new proof of the MacWilliams identity for the Hamming weight enumerator. 
\begin{ex}\label{ex:Hamming}
    Let $R$ be any finite commutative Frobenius ring, and let $\{a_0,a_1,\dots,a_t\}$ be the set of representatives described in Subsection~\ref{subsection:swe}. Without loss of generality, we choose $a_0:=0$ and $a_t:=1$ for the rest of this example. Now let $\mu$ and $\mathbf{A}$ be the M\"{o}bius function and the adjacency matrix of the poset of principal ideals of $R$ under set inclusion, respectively. Since $\mathbf{A}^{-1}=(b_{ij})$ satisfies $b_{ij}=\mu(a_i R, a_j R)$ for any $i,j\in\{0,1,\dots,t\}$, it follows that 
\begin{equation}\label{eq:mobius-1}
   \mu(a_k R, a_0 R)=
    \begin{cases}
    1,&\text{ if }k=0\,;\\
    0,&\text{ otherwise};
    \end{cases}  
\end{equation}
\begin{equation}\label{eq:mobius-2}
   \sum_{j=0}^t\mu(a_k R, a_j R)=
    \begin{cases}
    1,&\text{ if }k=t\,;\\
    0,&\text{ otherwise}.
    \end{cases}  
\end{equation}
Let $\mathbf{Q}=(q_{ij})$ and $\mathbf{D}$ be the $(t+1)\times (t+1)$ matrices described in Theorem~\ref{thm:main} and define $\mathbf{S}=(s_{ij})$ by $\mathbf{S}:=\mathbf{QDA}^{-1}$. We will now use Theorem~\ref{thm:main} to prove the MacWilliams identity for the Hamming weight enumerator. From (\ref{eq:HWE-SWE}) and Theorem~\ref{thm:main},
\[W_{C^\perp}(x,y)=\swe_{C^\perp}(x,y,\dots,y)=\frac{1}{|C|}\swe_C(z_0,z_1,\dots,z_t),\]
where for any $i\in\{0,1,\dots,t\}$,
\begin{equation}\label{eq:HWE-proof}
    z_i=s_{i\hspace*{.3mm}0}\,x+\sum_{j=1}^t s_{ij}\,y\,.
\end{equation}
Observe that from (\ref{eq:Q-alt}) and (\ref{eq:mobius-1}),
\begin{equation}\label{eq:HWE-proof-1}
    s_{i\hspace*{.3mm}0}=\sum_{k=0}^t q_{ik}\;|a_k R|\;\mu(a_k R, a_0 R)=q_{i\hspace*{.3mm}0}\;|a_0 R|=1\,.
\end{equation}
Moreover, we can use Equations (\ref{eq:Q-alt}) and (\ref{eq:mobius-2}) to obtain
\begingroup
	\allowdisplaybreaks
\begin{align*}
    \sum_{j=0}^t s_{ij}&=\sum_{j=0}^t\sum_{k=0}^t q_{ik}\;|a_k R|\;\mu(a_k R, a_j R)\\
    &=\sum_{k=0}^t q_{ik}\;|a_k R|\sum_{j=0}^t\mu(a_k R, a_j R)\\
    &=q_{i\hspace*{.2mm}t}\;|a_t R|\\
    &=\begin{cases}
    |R|,&\text{ if }i=0\,;\\
    \;\;0\,,&\text{ otherwise}.
    \end{cases}
\end{align*}
\endgroup
Therefore, by (\ref{eq:HWE-proof-1}),
\begin{equation}\label{eq:HWE-proof-2}
    \sum_{j=1}^t s_{ij}=-s_{i\hspace*{.3mm}0}+\sum_{j=0}^t s_{ij}=\begin{cases}
    |R|-1,&\text{ if }i=0\,;\\
    \quad\;\,\,\,-1,&\text{ otherwise}.
    \end{cases}
\end{equation}
Substituting (\ref{eq:HWE-proof-1}) and (\ref{eq:HWE-proof-2}) to (\ref{eq:HWE-proof}) gives us
\[z_i=\begin{cases}
    x+(|R|-1)y,&\text{ if }i=0\,;\\
    \quad\quad\quad\;\,\,\,x-y,&\text{ otherwise.}
    \end{cases}\]
The desired MacWilliams identity (\ref{eq:MW-HWE}) follows from this.
\end{ex}

\section{Generalizations to supports and tuples}\label{sec:generalizations}
In this section, we generalize the symmetrized weight enumerator with respect to symmetrized supports and tuples of codewords, and prove three MacWilliams identities that generalize the MacWilliams identity by Wood \cite{Wood1999} (Theorem~\ref{thm:MW-Wood}). Most of the proofs will be very similar to the arguments in Section 3, so we will omit some of the more technical details. 

First, we define the symmetrized support enumerator of linear codes over finite commutative Frobenius rings. Let $C$ be a linear code over a finite commutative Frobenius ring $R$. Throughout this section, let
\[\mathbf{X}:=\begin{pmatrix}
    x_{01} &\cdots &x_{0n}\\
    x_{11} &\cdots &x_{1n}\\
    \vdots &\ddots &\vdots\\ 
    x_{t1} &\cdots &x_{tn}
\end{pmatrix} \qquad\text{and}\qquad
\mathbf{Y}:=\begin{pmatrix}
    y_{01} &\cdots &y_{0n}\\
    y_{11} &\cdots &y_{1n}\\
    \vdots &\ddots &\vdots\\ 
    y_{t1} &\cdots &y_{tn}
\end{pmatrix}\]
be matrices of indeterminates whose rows and columns are indexed by $\{0,1,\dots,t\}$ and $[n]$, respectively. 
The \emph{symmetrized support enumerator} of $C$ is
\begin{equation}\label{eq:SSE-multiset}
    \sse_C(\mathbf{X}):=\sum_{X\subseteq t\cdot[n]} A_C(X) \Big(\prod_{\ell\in M_0(X)} x_{0\ell}\prod_{\ell\in M_1(X)} x_{1\ell}\cdots \prod_{\ell\in M_t(X)} x_{t\ell}\Big)\,,
\end{equation}
where $A_C$ is as defined in (\ref{eq:AC}). This generalizes $\swe_C$ in (\ref{eq:SWE-multiset}) which is obtained from $\sse_C$ in~(\ref{eq:SSE-multiset}) by setting $x_{i1}=\cdots=x_{in}:=x_i$ for all $i=0,1,\dots,t$. In other words, $\sse_C$ gives more information than $\swe_C$ about the codewords of $C$.

As in Theorem~\ref{thm:id-SWE}, we can also replace $A_C$ in (\ref{eq:SSE-multiset}) by $B^{\mathcal{F}}_{C}$ and obtain a new expression in terms of $\sse_C$ as follows.
\begin{theorem}\label{thm:id-SSE}
    Let $C$ be a linear code of length $n$ over a finite commutative Frobenius ring~$R$, and let $\mathcal{F}:=(F_0,F_1,\dots,F_t)$, where $F_0,F_1,\dots,F_t$ are non-empty subsets of $\{0,1,\dots,t\}$. Then
    \begin{equation}\label{eq:id-SSE}
        \sse_C\big(\mathbf{P}^\mathcal{F}\mathbf{Y}\big)=\sum_{Y\subseteq t\cdot[n]} B^{\mathcal{F}}_{C}(Y)\Big(\prod_{\ell\in M_0(Y)} y_{0\ell}\prod_{\ell\in M_1(Y)} y_{1\ell}\cdots\prod_{\ell\in M_t(Y)} y_{t\ell}\Big)\,,
    \end{equation}
    where $\mathbf{P}^\mathcal{F}$ is the $(t+1)\times (t+1)$ matrix described in Theorem~\ref{thm:id-SWE}.
\end{theorem}
\begin{proof}
    The proof is similar to Theorem~\ref{thm:id-SWE} but we use more general identities. In particular, Equation (\ref{eq:id-part}) becomes
    \[\sum_{(T_1,\dots,T_m)\in\mathcal{P}_m(T)} \Big(\prod_{\ell\in T_1} u_{1\ell}\cdots \prod_{\ell\in T_m} u_{m\ell}\Big)=\prod_{\ell\in T}(u_{1\ell}+\cdots+u_{m\ell})\]
    and (\ref{eq:proof2}) becomes
    \[M_j(Y)=\bigcup_{\substack{i\in\{0,1,\dots,t\}\colon\\j\in F_i}}\hspace*{-2mm}\!\! W_{ij}.\]
    The rest of the proof follows from these two more general identities.
\end{proof}
\noindent We hereby obtain the following MacWilliams identity for the symmetrised support enumerator. The proof is very similar to that of Theorem~\ref{thm:main}, but here we use Theorem~\ref{thm:id-SSE} instead of Theorem~\ref{thm:id-SWE}.
\begin{theorem}\label{thm:main-support}
    Let $C$ be a linear code of length $n$ over a finite commutative Frobenius ring~$R$. Then 
    \[\sse_{C^\perp}(\mathbf{X})=\frac{1}{|C|}\sse_C\big(\mathbf{QDA}^{-1}\mathbf{X}\big)\]
    for the $(t+1)\times (t+1)$ matrices $\mathbf{Q, D, A}$ described in Theorem~\ref{thm:main}.
\end{theorem}
We now generalize $\sse_C$ further by considering tuples of codewords instead of single codewords, as in \cite{BrShWe15,shiromoto96}. However, we will restrict the focus to finite principal ideal rings~$R$. Let $C$ be a linear code of length $n$ over $R$, and let $\lambda$ be a positive integer, and let $\cv_m=(c_{m1},\dots,c_{mn})\in C$ for each $m\in\{1,\dots,\lambda\}$. For each integer $i\in\{0,1,\dots,t\}$, define
\[S^{[\lambda]}_i(\cv_1,\dots,\cv_\lambda):=\{\ell\in [n]\;\colon\;c_{1\ell} R+\cdots+c_{\lambda \ell} R=a_i R\}\,.\]
Here, $c_{1\ell} R+\cdots+c_{\lambda \ell} R$ is the ideal of $R$ generated by $c_{1\ell},\dots,c_{\lambda \ell}$. Since $R$ is a principal ideal ring, we know that $S^{[\lambda]}_0(\cv_1,\dots,\cv_\lambda)$, $S^{[\lambda]}_1(\cv_1,\dots,\cv_\lambda),\dots, S^{[\lambda]}_t(\cv_1,\dots,\cv_\lambda)$ are disjoint sets whose union is $[n]$. The \emph{$\lambda$-tuple symmetrized support enumerator} of $C$ is
\begin{equation}\label{eq:SSE-tuple-multiset}
    \sse^{[\lambda]}_C(\mathbf{X}):=\sum_{X\subseteq t\cdot[n]} A^{[\lambda]}_C(X) \Big(\prod_{\ell\in M_0(X)} x_{0\ell}\prod_{\ell\in M_1(X)} x_{1\ell}\cdots \prod_{\ell\in M_t(X)} x_{t\ell}\Big),
\end{equation}
where for any multiset $X\subseteq t\cdot[n]$,
\[A^{[\lambda]}_C(X):=|\{(\cv_1,\dots,\cv_\lambda)\in C^\lambda\;\colon\; S^{[\lambda]}_i(\cv_1,\dots,\cv_\lambda)=M_i(X) \text{ for all }i=0,1,\dots,t\}|\,.\] 
This generalizes $\sse_C$ in (\ref{eq:SSE-multiset}) which is a special case of $\sse^{[\lambda]}_C$ in (\ref{eq:SSE-tuple-multiset}) when $\lambda=1$.

In this generalization, we can replace $A^{[\lambda]}_C$ in (\ref{eq:SSE-tuple-multiset}) by $(B^{\mathcal{I}}_{C})^\lambda$, where $\mathcal{I}$ is as defined in Lemma~\ref{lem:id-SWE}, and obtain a new expression in terms of $\sse^{[\lambda]}_C$. However, unlike Theorems~\ref{thm:id-SWE} and \ref{thm:id-SSE}, this is not always true for other choices of $\mathcal{F}$. 
\begin{theorem}\label{thm:id-SSE-tuple}
    Let $C$ be a linear code of length $n$ over a finite commutative principal ideal ring~$R$, and let $\{a_0,a_1,\dots,a_t\}$ be the set of representatives described in Subsection~\ref{subsection:swe}. Let $\mathcal{I}:=(I_0,I_1,\dots,I_t)$ be a tuple of subsets of $\{0,1,\dots,t\}$, where for each $i\in\{0,1,\dots,t\}$,
    \begin{align*}
    I_i&:=\{j\in\{0,1,\dots,t\}\;\colon\;a_j a_k\neq 0\text{ for all }k\text{ such that }a_ia_k\neq 0\}\,.
    \end{align*}
    Then for any positive integer $\lambda$,
    \begin{equation*}\label{eq:id-SSE-tuple}
        \sse^{[\lambda]}_C(\mathbf{A}\mathbf{Y})=\sum_{Y\subseteq t\cdot[n]}\!\big(B^{\mathcal{I}}_{C}(Y)\big)^\lambda\Big(\prod_{\ell\in M_0(Y)} y_{0\ell}\prod_{\ell\in M_1(Y)} y_{1\ell}\cdots\prod_{\ell\in M_t(Y)} y_{t\ell}\Big)\,,
    \end{equation*}
    where $\mathbf{A}$ is the adjacency matrix of the poset of principal ideals of $R$ under set inclusion. 
\end{theorem}
\begin{proof}
    The proof is similar to that of Theorem~\ref{thm:id-SSE} when $\mathcal{F}=\mathcal{I}$, but here we use 
    \begin{equation}\label{eq:id-BICL}
        \big(B^{\mathcal{I}}_{C}(Y)\big)^\lambda=\sum_{X\leq Y}A^{[\lambda]}_C(X)
    \end{equation}
    and the fact that $\mathbf{A}:=\mathbf{P}^{\mathcal{I}}$ is the adjacency matrix of the poset of principal ideals of $R$ under set inclusion; see Proposition~\ref{prop:A}. We will only prove that (\ref{eq:id-BICL}) holds and the rest of the proof follows. From the equivalence of (A1) and (A2) in the proof of Lemma~\ref{lem:id-SWE}, we know that for any multiset $Y\subseteq t\cdot[n]$,
    \[B^{\mathcal{I}}_C(Y)=\Big|\Big\{\cv\in C\;\colon\;\supp(a_k\cv)\subseteq\bigcup_{\substack{j\in\{0,1,\dots,t\}\colon\\a_j a_k\neq 0}}\hspace*{-2mm}M_j(Y)\,\text{ for all }k=0,1,\dots,t\Big\}\Big|\,,\]
    and thus
    \[\big(B^{\mathcal{I}}_C(Y)\big)^\lambda=\Big|\Big\{(\cv_1,\dots,\cv_\lambda)\in C^\lambda\,\colon\hspace*{-1mm}\bigcup_{m=1}^\lambda\supp(a_k\cv_m)\subseteq\hspace*{-.5mm}\bigcup_{\substack{j\in\{0,1,\dots,t\}\colon\\a_j a_k\neq 0}}\hspace*{-2mm}\!\! M_j(Y)\,\text{ for all }k=0,1,\dots,t\Big\}\Big|\,.\]
    Let $\leq$ be the binary relation on $t\cdot[n]$ defined in the proof of Theorem~\ref{thm:id-SWE}. We can then prove (\ref{eq:id-BICL}) by showing that
    \begin{align}\label{eq:proof-BICL}
        &\bigcup_{X\leq Y}\{(\cv_1,\dots,\cv_\lambda)\in C^\lambda\;\colon\; S^{[\lambda]}_i(\cv_1,\dots,\cv_\lambda)=M_i(X) \text{ for all }i=0,1,\dots,t\}\notag\\
        &=\Big\{(\cv_1,\dots,\cv_\lambda)\in C^\lambda\,\colon\hspace*{-1mm}\bigcup_{m=1}^\lambda\supp(a_k\cv_m)\subseteq\hspace*{-.5mm}\bigcup_{\substack{j\in\{0,1,\dots,t\}\colon\\a_j a_k\neq 0}}\hspace*{-2mm}\!\! M_j(Y)\,\text{ for all }k=0,1,\dots,t\Big\}\,.
    \end{align}
    Consider any multiset $X\leq Y$ and $(\cv_1,\dots,\cv_\lambda)\in C^\lambda$ such that \[S^{[\lambda]}_j(\cv_1,\dots,\cv_\lambda)=M_j(X)\subseteq \displaystyle\bigcup_{k\in I_j}M_k(Y)\] for all $j=0,1,\dots,t$. For any $i\in\{0,1,\dots,t\}$ and $m\in\{1,\dots,\lambda\}$, it is not hard to show that
    \[S_i(\cv_m)\,\subseteq\,\bigcup_{\substack{j\in\{0,1,\dots,t\}\colon\\ a_i R\subseteq a_j R\\}}\hspace*{-2mm}\!\! S^{[\lambda]}_j(\cv_1,\dots,\cv_\lambda)\,=\,\bigcup_{j\in I_i} S^{[\lambda]}_j(\cv_1,\dots,\cv_\lambda)\,\subseteq\,\bigcup_{j\in I_i}\bigcup_{k\in I_j} M_k(Y)\,=\,\bigcup_{j\in I_i}M_j(Y).\]
    The equivalence of (A1) and (A2) in the proof of Lemma~\ref{lem:id-SWE} implies that $(\cv_1,\dots,\cv_\lambda)$ is also an element of the right-hand set in (\ref{eq:proof-BICL}). Now consider the converse: let $(\cv_1,\dots,\cv_\lambda)\in C^\lambda$ be an element of the right-hand set in (\ref{eq:proof-BICL}). One can show that for any $i\in\{0,1,\dots,t\}$,
    \[S^{[\lambda]}_i(\cv_1,\dots,\cv_\lambda)\subseteq \bigcap_{\substack{k\in\{0,1,\dots,t\}\colon\\a_i a_k\neq 0}}\bigcup_{m=1}^\lambda \supp(a_k\cv_m)\subseteq \bigcap_{\substack{k\in\{0,1,\dots,t\}\colon\\a_i a_k\neq 0}}\bigcup_{\substack{j\in\{0,1,\dots,t\}\colon\\a_j a_k\neq 0}}\hspace*{-2mm}\!\! M_j(Y)=\bigcup_{j\in I_i}M_j(Y)\]
    and this completes the proof of (\ref{eq:proof-BICL}). Therefore, (\ref{eq:id-BICL}) holds.
\end{proof}
We obtain the MacWilliams identity for the $\lambda$-tuple symmetrized support weight enumerator as follows.
\begin{theorem}\label{thm:main-tuple}
    Let $C$ be a linear code of length $n$ over a finite commutative principal ideal ring $R$. Then 
    \[\sse^{[\lambda]}_{C^\perp}(\mathbf{X})=\frac{1}{|C|^\lambda}\sse^{[\lambda]}_C\big(\mathbf{QD}^\lambda \mathbf{A}^{-1}\mathbf{X}\big)\]
    for the $(t+1)\times (t+1)$ matrices $\mathbf{Q, D, A}$ described in Theorem~\ref{thm:main}. Consequently,
    \begin{equation}\label{eq:matrix-involution}
        \big(\mathbf{QD}^\lambda \mathbf{A}^{-1}\big)^2=|R|^\lambda\;\mathbf{I},
    \end{equation}
    where $\mathbf{I}$ is the $(t+1)\times (t+1)$ identity matrix.
\end{theorem}
\begin{proof}
    Applying Theorem~\ref{thm:id-SSE-tuple} to $C^\perp$ gives us
    \begin{equation}\label{eq:tuple-proof-1}
        \sse^{[\lambda]}_{C^\perp}(\mathbf{X})=\sum_{Y\subseteq t\cdot[n]} \big(B^{\mathcal{I}}_{C^\perp}(Y)\big)^\lambda\prod_{j=0}^t \Big(\prod_{\ell\in M_j(Y)} y_{j\ell}\Big)\,,
    \end{equation}
    where $\mathbf{X}=\mathbf{AY}$ and $\mathbf{A}$ is the adjacency matrix of the poset of principal ideals of $R$ under set inclusion. Therefore, $\mathbf{Y}=\mathbf{A}^{-1}\mathbf{X}$. Since $R$ is a principal ideal ring, there exists an involution~$\phi$ on $\{0,1,\dots,t\}$ such that $(a_j R)^\perp=a_{\phi(j)}R$ for all $j=0,1,\dots,t$. From the statements (a) and (d) in the proof of Theorem~\ref{thm:main} and statements (a) and (d) in the proof of Proposition~\ref{prop:A}, we deduce that 
    \[J_i=\{j\in\{0,1,\dots,t\}\;\colon\;a_j R\subseteq (a_i R)^\perp\}=\{\phi(j)\;\colon\;j\in I_i\}\]
    for any $i\in\{0,1,\dots,t\}$. Define the function $\Phi$ on the (multiset) power set of $t\cdot[n]$ such that for each  multiset $X\subseteq t\cdot[n]$, $\Phi(X)\subseteq t\cdot[n]$ is the multiset with $M_i(\Phi(X))=M_{\phi(i)}(X)$ for all $i\in\{0,1,\dots,t\}$. Note that $\Phi$ is an involution. Then for any multiset $Y\subseteq t\cdot[n]$, one can see that
    \begin{equation}\label{eq:tuple-proof-2}
        B^{\mathcal{J}}_C(Y)=B^{\mathcal{I}}_C(\Phi(Y))\,.
    \end{equation}
    From Lemma~\ref{lem:id-SWE} and (\ref{eq:tuple-proof-2}),
    \begingroup
	\allowdisplaybreaks
    \begin{align*}
        \sum_{Y\subseteq t\cdot[n]} \big(B^{\mathcal{I}}_{C^\perp}(Y)\big)^\lambda\prod_{j=0}^t \Big(\prod_{\ell\in M_j(Y)} y_{j\ell}\Big)
        &=\sum_{Y\subseteq t\cdot[n]} \frac{\prod_{i=0}^t|a_iR|^{\,\lambda\,|M_i(Y)|}}{|C|^\lambda}B^{\mathcal{J}}_C(Y)\prod_{j=0}^t \Big(\prod_{\ell\in M_j(Y)} y_{j\ell}\Big)\\
        &=\frac{1}{|C|^\lambda}\sum_{Y\subseteq t\cdot[n]}B^{\mathcal{J}}_C(Y)\prod_{j=0}^t \Big(\prod_{\ell\in M_j(Y)} |a_j R|^\lambda y_{j\ell}\Big)\\
        &=\frac{1}{|C|^\lambda}\sum_{Y\subseteq t\cdot[n]}B^{\mathcal{I}}_C(\Phi(Y))\prod_{j=0}^t \Big(\prod_{\ell\in M_j(Y)} |a_j R|^\lambda y_{j\ell}\Big)\\
        &=\frac{1}{|C|^\lambda}\sum_{Y\subseteq t\cdot[n]}B^{\mathcal{I}}_C(Y)\prod_{j=0}^t\Big(\prod_{\ell\in M_j(Y)} w_{j\ell}\Big),
    \end{align*}
    \endgroup
    where $\mathbf{W}=(w_{j\ell})$ is the matrix such that  $\mathbf{W}=\mathbf{P}\mathbf{D}^\lambda\mathbf{Y}=\mathbf{PD}^\lambda\mathbf{A}^{-1}\mathbf{X}$, and where $\mathbf{P}$ is the permutation matrix with respect to $\phi^{-1}=\phi$. Combine the above equation with (\ref{eq:tuple-proof-1}) and then apply Theorem~\ref{thm:id-SSE-tuple} again to obtain
    \[\sse^{[\lambda]}_{C^\perp}(\mathbf{X})=\frac{1}{|C|^\lambda}\sse^{[\lambda]}_C(\mathbf{AW})\,.\]
    Here, $\mathbf{AW}=\mathbf{APD}^\lambda \mathbf{A}^{-1} \mathbf{X}=\mathbf{QD}^\lambda \mathbf{A}^{-1} \mathbf{X}$ since $\mathbf{Q}=\mathbf{AP}$ from Corollary~\ref{cor:FPIR}. Equation (\ref{eq:matrix-involution}) can be obtained from Lemmas~\ref{lem:Wood} and \ref{lem:AA} after applying the MacWilliams identity in this theorem to $C$ and $C^\perp$, respectively.
\end{proof}

\begin{rem}
    The MacWilliams identity in Theorem~\ref{thm:main-tuple} does not hold for general finite Frobenius rings since  (\ref{eq:matrix-involution}) might not hold. For instance, the ring $R:=\FF_2[u,v]/\langle u^2,v^2\rangle$ from Example~\ref{ex:F2[u,v]} satisfies $\big(\mathbf{QD A}^{-1}\big)^2=|R|\;\mathbf{I}$ but $\big(\mathbf{QD}^2\mathbf{A}^{-1}\big)^2\neq |R|^2\;\mathbf{I}$.
\end{rem}

\noindent For finite chain rings, we can find the explicit form of the transformation matrix $\mathbf{QD}^\lambda\mathbf{A}^{-1}$. We refer to Subsection \ref{subsection:FCR} for the definitions of $q$ and $\nu$.
\begin{cor}\label{cor:FCR-tuple}
    Let $C$ be a linear code over a finite commutative chain ring $R$.
    Then for any positive integer $\lambda$,
    \[
    \sse^{[\lambda]}_{C^\perp}(\mathbf{X})=\frac{1}{|C|^\lambda}\sse^{[\lambda]}_C\big(\mathbf{S}^{[\lambda]}\mathbf{X}\big),
    \]
    where
    \[\mathbf{S}^{[\lambda]}=\mathbf{QD}^\lambda\mathbf{A}^{-1}=\begin{pmatrix}
    1 & q^\lambda-1 & q^{2\lambda}-q^\lambda &\cdots &q^{\nu\lambda}-q^{(\nu-1)\lambda} \\
    1 & q^\lambda-1 & q^{2\lambda}-q^\lambda &\cdots &-q^{(\nu-1)\lambda}\\
    \vdots &\vdots &\vdots &\iddots  &\vdots\\
    1 & q^\lambda-1 &-q^\lambda &\cdots &0\\
    1 &-1 &0 &\cdots &0
\end{pmatrix}.\]
\end{cor}

Now define the \emph{$\lambda$-tuple symmetrized weight enumerator} of $C$ as
\begin{equation}\label{eq:SWE-tuple}
    \swe^{[\lambda]}_C(\xv)=\swe^{[\lambda]}_C(x_0,x_1,\dots,x_t)
    :=\sum_{(\cv_1,\dots,\cv_\lambda)\in C^\lambda}\!\! x_0^{|S^{[\lambda]}_0(\cv_1,\dots,\cv_\lambda)|}x_1^{|S^{[\lambda]}_1(\cv_1,\dots,\cv_\lambda)|}\cdots x_t^{|S^{[\lambda]}_t(\cv_1,\dots,\cv_\lambda)|}\,.
\end{equation}
\noindent This enumerates the codeword $\lambda$-tuples based on their number of coordinates that generate each principal ideal of $R$. Note that $\swe^{[\lambda]}_C$ in (\ref{eq:SWE-tuple}) can be obtained from $\sse^{[\lambda]}_C$ in (\ref{eq:SSE-tuple-multiset}) by setting $x_{i1}=\cdots=x_{in}:=x_i$ for all $i=0,1,\dots,t$. The following MacWilliams identity therefore holds for the $\lambda$-tuple symmetrized weight enumerator by Theorem~\ref{thm:main-tuple}.
\begin{cor}\label{cor:SWE-tuple}
    Let $C$ be a linear code of length $n$ over a finite commutative principal ideal ring~$R$. Then 
    \[\swe^{[\lambda]}_{C^\perp}(\xv)=\frac{1}{|C|^\lambda}\swe^{[\lambda]}_C\big(\mathbf{QD}^\lambda \mathbf{A}^{-1}\xv\big)\]
    for the $(t+1)\times (t+1)$ matrices $\mathbf{Q, D, A}$ described in Theorem~\ref{thm:main}.
\end{cor}
\begin{rem}
    We note that $\swe^{[\lambda]}_C$ in (\ref{eq:SWE-tuple}) is a special case of the $r$-fold symmetric weight enumerator considered by Siap \cite{Siap1999}. However, $\sse_C$ in~(\ref{eq:SSE-multiset}) and $\sse^{[\lambda]}_C$ in (\ref{eq:SSE-tuple-multiset}) are neither generalizations nor special cases of the enumerators considered in \cite{Siap1999}.
\end{rem}
Each of the MacWilliams identities in Theorems~\ref{thm:main-support} and \ref{thm:main-tuple} and Corollary~\ref{cor:SWE-tuple} generalizes the classic MacWilliams identity for the symmetrized weight enumerator by Wood \cite{Wood1999} (Theorem~\ref{thm:MW-Wood}). Moreover, Theorem~\ref{thm:main-tuple} generalizes \cite[Theorem 6]{BrShWe15}. 

We close this section with an example.
\begin{ex}
    Consider the case when $R=\ZZ_{6}:=\{0,1,2,3,4,5\}$ is the ring of integers modulo 6 
    and choose $a_0=0$, $a_1=3$, $a_2=2$, $a_3=1$. 
    The lattice of principal ideals of $\ZZ_{6}$ under set inclusion is presented in Figure~\ref{fig:Z6} below.
\begin{figure}[ht]
    \centering
    \begin{tikzpicture}
        \draw (0,0) -- (1,1) -- (0,2);
        \draw (0,0) -- (-1,1) -- (0,2);
        \draw (0,0) node [fill=white] {$a_0 R$};
        \draw (-1,1) node [fill=white] {$a_1 R$};
        \draw (1,1) node [fill=white] {$a_2 R$};
        \draw (0,2) node [fill=white] {$a_3 R$};
    \end{tikzpicture}
    \caption{Hasse diagram for principal ideals of $\ZZ_{6}$.}
    \label{fig:Z6}
\end{figure}
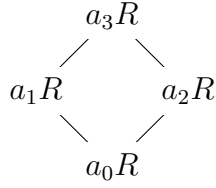

The matrices $\mathbf{A, D, Q}$ according to Theorem~\ref{thm:main} are
    \[\mathbf{A}=\begin{pmatrix}
    1 & 1 & 1 &1\\
    0 & 1 & 0 &1\\
    0 & 0 & 1 &1\\
    0 & 0 & 0 &1
\end{pmatrix},\quad
\mathbf{D}=\begin{pmatrix}
    1 & 0 & 0 & 0\\
    0 & 2 & 0 & 0\\
    0 & 0 & 3 & 0\\
    0 & 0 & 0 & 6
\end{pmatrix},\quad
\mathbf{Q}=\begin{pmatrix}
    1 & 1 & 1 &1\\
    1 & 0 & 1 &0\\
    1 & 1 & 0 &0\\
    1 & 0 & 0 &0
\end{pmatrix}.\]
The transformation matrix with respect to Theorem~\ref{thm:main-tuple} for $\lambda=1$ and $\lambda=2$ respectively are
\[\mathbf{QDA}^{-1}=
\begin{pmatrix*}[r]
    1 & 1 & 2 & 2\\
    1 & -1 & 2 & -2\\
    1 & 1 & -1 &-1\\
    1 & -1 & -1 & 1
\end{pmatrix*}
\quad\text{and}\quad
\mathbf{QD}^2 \mathbf{A}^{-1}=
\begin{pmatrix*}[r]
    1 & 3 & 8 & 24\\
    1 & -1 & 8 & -8\\
    1 & 3 & -1 &-3\\
    1 & -1 & -1 & 1
\end{pmatrix*}.\]
Now consider the linear code $C=\{(0,0),(1,4),(2,2),(3,0),(4,4),(5,2)\}$ of length 2 over $R$. Here, $C^\perp=\{(0,0),(2,1),(4,2),(0,3),(2,4),(4,5)\}$. By definition (\ref{eq:SSE-tuple-multiset}),
\[
  \begin{array}{l@{\!}l@{\;\,}c@{\;\,}l@{\;}c@{\;}r@{\;}c@{\;}r@{\;}c@{\;}r}
    \sse^{[1]}_C&(\mathbf{X})         &=&x_{01} x_{02}&+& x_{02} x_{11}&+&2x_{21}x_{22}&+& 2x_{22}x_{31}\,;\\[.5mm]
    \sse^{[1]}_{C^\perp}&(\mathbf{X}) &=&x_{01} x_{02}&+& x_{01} x_{12}&+&2x_{21}x_{22}&+& 2x_{21}x_{32}\,;\\[.5mm]
    \sse^{[2]}_C&(\mathbf{X})         &=&x_{01} x_{02}&+&3x_{02} x_{11}&+&8x_{21}x_{22}&+&24x_{22}x_{31}\,;\\[.5mm]
    \sse^{[2]}_{C^\perp}&(\mathbf{X}) &=&x_{01} x_{02}&+&3x_{01} x_{12}&+&8x_{21}x_{22}&+&24x_{21}x_{32}\,.
  \end{array}
\]
We can check that
\[
  \begin{array}{l@{\;\,}c@{\;\,}l@{\!}r@{\;}c@{\;}r@{\;}c@{\;}r@{\;}c@{\;}r@{\;}c@{\;}r@{\;}c@{\;}r@{\;}c@{\;}r@{\;}r@{\,}r}
     \frac{1}{6}\sse^{[1]}_C\big(\mathbf{QDA}^{-1}\mathbf{X}\big)
    &=&\frac{1}{6}\Big(& (x_{01}&+&x_{11}&+&2x_{21}&+&2x_{31})&(x_{02}&+&x_{12}&+&2x_{22}&+&2x_{32})&\\
    & &               +& (x_{02}&+&x_{12}&+&2x_{22}&+&2x_{32})&(x_{01}&-&x_{11}&+&2x_{21}&-&2x_{31})&\\
    & &               +&2(x_{01}&+&x_{11}&-& x_{21}&-& x_{31})&(x_{02}&+&x_{12}&-& x_{22}&-& x_{32})&\\
    & &               +&2(x_{02}&+&x_{12}&-& x_{22}&-& x_{32})&(x_{01}&-&x_{11}&-& x_{21}&+& x_{31})&\Big)\\
    &=&\sse^{[1]}_{C^\perp}(\mathbf{X})\;,\hspace*{-20mm}
\end{array}
\]
and
\[
  \begin{array}{l@{\;\,}c@{\;\,}l@{\!\!\!}r@{\;}c@{\;}r@{\;}c@{\;}r@{\;}c@{\;}r@{\;}c@{\;}r@{\;}c@{\;}r@{\;}c@{\;}r@{\;}r@{\;}r@{\;}c@{\!\!}r}
      \frac{1}{36}\sse^{[2]}_C\big(\mathbf{QD}^2\mathbf{A}^{-1}\mathbf{X}\big)
    &=&\frac{1}{36}\Big(&  (x_{01}&+&3x_{11}&+&8x_{21}&+&24x_{31})&(x_{02}&+&3x_{12}&+&8x_{22}&+&24x_{32})&\\
    & &                +& 3(x_{02}&+&3x_{12}&+&8x_{22}&+&24x_{32})&(x_{01}&-& x_{11}&+&8x_{21}&-& 8x_{31})&\\
    & &                +& 8(x_{01}&+&3x_{11}&-& x_{21}&-& 3x_{31})&(x_{02}&+&3x_{12}&-& x_{22}&-& 3x_{32})&\\
    & &                +&24(x_{02}&+&3x_{12}&-& x_{22}&-& 3x_{32})&(x_{01}&-& x_{11}&-& x_{21}&+&  x_{31})&\Big)\\
    &=&\sse^{[2]}_{C^\perp}(\mathbf{X})\;,\hspace*{-20mm}
  \end{array}
\]
as stated in Theorem~\ref{thm:main-tuple}. Note that the case when $\lambda=1$ is exactly Theorem~\ref{thm:main-support}. We can also easily check that Corollary~\ref{cor:SWE-tuple} holds by setting $x_{i1}=x_{i2}:=x_i$ for all $i=0,1,2,3$.
\end{ex}


\end{document}